\documentclass[ leqno,final]{article}

\usepackage{amsfonts,amsmath,amsthm}
\usepackage{epsfig,graphicx,showkeys,color,a4wide}
\def\R{\mathbb{R}}

\def\N{\mathbb{N}}

\def\e{\varepsilon}
\def\a{\alpha}

\def \vs{\vspace*{0.2cm}}
\newcommand{\Lip}{{\textrm{Lip}}}

\newtheorem{theo}{Theorem}[section]
\newtheorem{lem}[theo]{Lemma}
\newtheorem{pro}[theo]{Proposition}
\newtheorem{cor}[theo]{Corollary}

\theoremstyle{definition}
\newtheorem{defi}[theo]{Definition}
\newtheorem{rem}[theo]{Remark}

\newcommand{\ba}{\begin{eqnarray}}
\newcommand{\ea}{\end{eqnarray}}

\numberwithin{equation}{section}
\usepackage{color}

\begin{document}

\title{\bf Existence and uniqueness of traveling wave\\ for accelerated Frenkel-Kontorova model}
\date{2014}
\author{\renewcommand{\thefootnote}{\arabic{footnote}}
N. Forcadel\footnotemark[1],
A. Ghorbel\footnotemark[2]\-{$~^{,}$\footnotemark[3]},
S. Walha\footnotemark[1]$~^{,}$\footnotemark[2]}

\footnotetext[1]{INSA de Rouen, Normandie Universit\'e, Labo. de Math\'ematiques de l'INSA - LMI (EA 3226 - FR CNRS 3335)
685 Avenue de l'Universit\'e, 76801 St Etienne du Rouvray cedex.
France}
\footnotetext[2]{University of Sfax, Faculty of Sciences, Laboratory of  "Stability and control of systems, nonlinear PDEs", LR 13 ES 21.}
\footnotetext[3]{University of Sfax, Higher Institute of business administration of Sfax,
Airport Road Km 4, PB 1013, 3018 Sfax, Tunisia; e-mail: ghorbel@cermics.enpc.fr}

\maketitle

\vspace{5pt}

%%%%%%%%%%%%%%%%%%%%%%%%%%%%%%%%%%%%%%%%%%%%%%%%%%%%%%%%%%%%%%%%%%%%%%%%%%
%%%%%%%%%%%%%%%%%%%%%%%%%%%%%%%%%%%%%%%%%%%%%%%%%%%%%%%%%%%%%%%%%%%%%%%%%%
\begin{abstract}
In this paper, we study the existence and uniqueness of traveling wave solution for the accelerated Frenkel-Kontorova model. This model consists in a
 system of ODE that describes the motion particles in interaction. The most important applications we have in mind is the motion of crystal defects called dislocations. For this model, we prove the existence of 
traveling wave solutions under very weak assumptions. The uniqueness of the velocity is also studied as well as the uniqueness of the profile which used different types of strong maximum principle. As far as we know, this is the first result concerning traveling waves for accelerated, spatially discrete system.
\end{abstract}

\smallskip


%%%%%%%%%%%%%%%%%%%%%%%%%%%%%%%%%%%%%%%%%%%%%%%%%%%%%%%%%%%%%%%%%%%%%%%%%%
  % I N T R O D U C T I O N
%%%%%%%%%%%%%%%%%%%%%%%%%%%%%%%%%%%%%%%%%%%%%%%%%%%%%%%%%%%%%%%%%%%%%%%%%
%--------------------------------------------

      %SECTION 

%--------------------------------------------

\section{Introduction}
In the present paper, we study the accelerated Frenkel-Kontorova model (F-K) which describes a chain of particles interacting by an harmonic potential. Besides its
original aim of modeling crystal dislocations, the (F-K) model has many applications in physics such as the description of magnetic domain walls, atoms
adsorbed on a crystalline surface or superionic conductors (see for instance the book of Braun and Kivshar \cite{BK04} for an introduction to this model). The goal of this work is to prove the existence and the uniqueness of traveling wave as well as 
the uniqueness of the velocity. This work is a generalization of the one of Al Haj et al. \cite{AFM12} in which the authors study the fully overdamped case.

The study of traveling waves in reaction-diffusion equations has been introduced in pioneering works of Fisher \cite{r15} 
and Kolmogorov, Petrovsky and Piskunov \cite{KPP}.
Existence of traveling waves  solutions has been for instance obtained in \cite{AW,BNS,r26,K}.
More generally, there is a huge literature about existence, uniqueness and stability of traveling waves with various non linearities
with applications in particular in biology and combustion
and we refer for instance to the references cited in \cite{BH,r17}.
There are also several works on discrete or nonlocal versions of  reaction-diffusion equations
(see for instance \cite{BCC,r6,r13,C,CM-PS,r19,r23,r35,r38,r2,r3,WZ,r52,ZHH} and the references cited therein) and on damped hyperbolic equation (see \cite{DO,G1,G2,G3,H}) but, as far as we know, there is no result concerning hyperbolic discrete in space equations.
%-------------------------------
\subsection{{The Frenkel-Kontorova model }} 
The classical Frenkel-Kontorova (F-K) model describes the dynamics of crystal defects.
 If $ u_{i}(t) $ is the position of the particle $ i \in \mathbb{Z} $,
then the classical (F-K) models is given by the following dynamics   
$$ m_{0} \dfrac{d^{2}u_{i}}{dt^{2}}+ \dfrac{du_{i}}{dt}=u_{i+1}+u_{i-1}-2 u_{i}-\sin(2 \pi (u_{i}-L))-\sin(2\pi L) $$
where $  \dfrac{d^{2}u_{i}}{dt^{2}} $ denotes the acceleration of the $i$th particle, $ \dfrac{d u_{i} }{dt} $ is its  velocity, $m_{0}$
 denotes the mass of the particles, $-\sin(2\pi L)$ is a constant driving force which will cause the movement of the chain of atoms and $-\sin(2 \pi (u_{i}-L)) $  describes the force created by a 
periodic potential whose period is assumed to be $1$.  
We set
$$f_L(v)=-\sin(2 \pi (v-L))-\sin(2\pi L)$$
and  for all $ i \in \mathbb{Z} $ 
$$ \Xi_{i}(t)=u_{i}(t)+2 m_{0} \dfrac{du_{i}}{dt}(t). $$
We replace it in equation (\ref{eq 2}) in order to obtain the following monotone system: for $ i \in \mathbb{Z} $ and $ t \in (0,+\infty) $,
\begin{equation}\label{eq 1}
\left\lbrace
\begin{array}{lcl}
\dfrac{du_{i}}{dt}=\dfrac{1}{2m_{0}}(\Xi_{i}-u_{i}) \\
\dfrac{d\Xi_{i}}{dt}=2(u_{i+1}+u_{i-1}-2u_{i})+f_L(u_i)+\dfrac{1}{2m_{0}}(u_{i}-\Xi_{i}).
\end{array}
\right.
\end{equation} 
We look for particular traveling wave solution of (\ref{eq 1}), which have the form 
\begin{equation}\label{eq 2}
\left\lbrace
\begin{array}{lcl}
u_{i}(t)= \phi_{1}(i+c\,t) \\
\Xi_{i}(t)= \phi_{2}(i+c\,t). 
\end{array}
\right.
\end{equation}
If we replace (\ref{eq 2}) in (\ref{eq 1}), then the profile $ ( \phi_{1},\phi_{2}) $ should satisfy 
\begin{equation}\label{eq 3}
\left\lbrace
\begin{array}{lll}
 c \phi_{1}'(z) = \alpha_{0}(\phi_{2}(z)-\phi_{1}(z)) \\
 c \phi_{2}'(z)=2(\phi_{1}(z+1)+\phi_{1}(z-1)-2\phi_{1}(z))+2 f_L(\phi_{1}(z))+\alpha_{0}(\phi_{1}(z)-\phi_{2}(z)) 
\end{array}
\right.
\end{equation}
with $ z=i+c\,t$ and $\alpha_{0}=\dfrac{1}{2\,m_{0}}$. We then have the following result.
%---------------------------------------
\begin{theo}[Existence and uniqueness of traveling wave solution for Frenkel-Kontorova model]\label{th:0} 
There exists a constant $\alpha^*$ (which will be made precised later on assumption (A)) such that for all $\alpha_0\ge \alpha^*$, there exist a unique real $ c $  and two functions $\phi_{1}:\mathbb{R} \to \mathbb{R}$ and
$\phi_{2}:\mathbb{R} \to \mathbb{R}$ that satisfy 
\begin{equation}
\left\lbrace
\begin{array}{l}
 c \phi_{1}'(z) = \alpha_{0}(\phi_{2}(z)-\phi_{1}(z)) \\
 c \phi_{2}'(z)=2(\phi_{1}(z+1)+\phi_{1}(z-1)-2\phi_{1}(z))+2f_L(\phi_{1}(z))+\alpha_{0}(\phi_{1}(z)-\phi_{2}(z)) 
\\
\phi_1,\; \phi_2 \textrm{ are non-decreasing over }\R\\
\phi_{1}(-\infty)=0, \,\,\phi_{1}(+\infty)=1 \\
\phi_{2}(-\infty)=0, \,\,\phi_{2}(+\infty)=1 \\
\end{array}
\right.
\end{equation}
in the classical sense if $ c \neq 0 $ and almost everywhere if $ c=0 $.
Moreover,  if $c\ne 0$, then the two profiles are unique up to translation.
\end{theo}

\subsection {Main results for the general case}
We now consider a generalization of system \eqref{eq 3}. Given a function 
$ F: [0,1]^{N+1} \to \mathbb{R}$, we consider the system
$$\begin{cases}
 c\,\,\phi_{1}'(z) = \alpha_{0}(\phi_{2}(z)-\phi_{1}(z)) \\
 c\,\,\phi_{2}'(z)=2\,F((\phi_{1}(z+r_{i}))_{i=0,...,N})+\alpha_{0}(\phi_{1}(z)-\phi_{2}(z)).
\end{cases}$$
In order to provide our results, we introduce some assumptions on $ F $. 

\paragraph{Assumption (A)}
\begin{itemize}
\item \textbf{Regularity of $ F $: }  
 $ F $ is globally Lipschitz continuous over $[0,1]$;

\item \textbf{Monotonicity of $ F $:} 
$ F(X_{0},...,X_{N}) $ is non-decreasing in $X_{i} $ for $i \neq 0$, and 
\begin{equation}\label{eq:monV0}
 2 \dfrac{\partial F}{\partial X_{0}}+ \alpha_0> 0.
 \end{equation}
 \end{itemize}
We set $ F(v,....,v)=f(v)$.

\paragraph{Assumption (B)} 
\begin{itemize}
\item\textbf{Instability:} 
$f(0)=0=f(1)$ and there exists $b \in [0,1]$ such that 
$$f(b)=0,\quad f_{\mid (0,b)} < 0,\quad f_{\mid (b,1)} > 0 \quad \mbox{and} \quad f'(b) >0.$$
\item\textbf{Smoothness:} 
$ F $ is $ C^{1}$ in a neighborhood of $\{b\}^{N+1}$.
\end{itemize}\medskip

We give the first main result concerning the existence of traveling wave.
\begin{theo}[Existence of a traveling wave]\label{Th 1} 
Under assumptions $(A)$ and $(B)$, there exist a real $ c $  and two functions  $\phi_{1}:\,\,\mathbb{R} \to \mathbb{R}$ and
$\phi_{2}:\,\,\mathbb{R} \to \mathbb{R}$ that solves 
\begin{equation}\label{eq 4}
\left\lbrace
\begin{array}{lcl}
 c\,\,\phi_{1}'(z) = \alpha_{0}(\phi_{2}(z)-\phi_{1}(z)) \\
c\,\,\phi_{2}'(z)=2\,F((\phi_{1}(z+r_{i}))_{i=0,...,N})+\alpha_{0}(\phi_{1}(z)-\phi_{2}(z)) \\
\phi_1,\; \phi_2 \textrm{ are non-decreasing over }\R\\
\phi_{1}(-\infty)=0, \,\,\phi_{1}(+\infty)=1 \\
\phi_{2}(-\infty)=0, \,\,\phi_{2}(+\infty)=1 \\
 \end{array}
\right.
\end{equation}
in the classical sense if $ c \neq 0 $ and almost everywhere if $ c=0 $.
\end{theo}  
In order to prove the uniqueness of the traveling  waves, we require some additional assumptions:

\paragraph{Assumption (C): Inverse monotonicity close to $ \{ 0 \}^{N+1} $ and  $ E= \{1 \}^{N+1} $:}
There exists $ \beta_{0}>0 $ such that for $ a>0$, we have 
$$ \begin{cases}
F(X+(a,...,a)) < F(X)\,\,\mbox{ for\,\,all}\,\,X,X+(a,...,a)\,\in\,[0,\beta_{0}]^{N+1}\\
F(X+(a,...,a)) < F(X)\,\,\mbox{ for \,\,all}\,\,X,X+(a,...,a)\,\in\,[1-\beta_{0},1]^{N+1}\\    
\end{cases}$$
\paragraph{Assumption ($ D+$)}
\begin{itemize}
\item[i)] All the $r_{i}$ have the same sign: We assume that $r_{i} \le 0$, for\,\,all$\,\,i\,\in\,\,\{0,...,N\}$.
 \item[ii)] Strict monotonicity: F is increasing in $ X_{i^{+}} $ with $ r_{i^{+}} > 0 $.
\end{itemize}
\paragraph{Assumption ($ D-$)}
 \begin{itemize}
\item[i)] All the $r_{i}$ have the same sign: We assume that $r_{i} \ge 0$, for\,\,all$\,\,i\,\in\,\,\{0,...,N\}$.
\item[ii)]Strict monotonicity: F is increasing in $ X_{i^{-}} $ with $ r_{i^{-}} < 0 $.
 \end{itemize}\bigskip

We give the second main results concerning the uniqueness of the velocity and of the profile.
\begin{theo}[Uniqueness of the velocity and of the profile]\label{th 2bis} 
 We assume $(A)$ and let $(c, (\phi_{1}, \phi_{2})) $, with $\phi_1,\phi_2:\R\to [0,1]$, be a solution of  
\begin{equation}\label{eq 5}
\left\lbrace
\begin{array}{lcl}
 c\,\,\phi_{1}'(z) = \alpha_{0}(\phi_{2}(z)-\phi_{1}(z)) \\
c\,\,\phi_{2}'(z)=2\,F((\phi_{1}(z+r_{i}))_{i=0,...,N})+\alpha_{0}(\phi_{1}(z)-\phi_{2}(z)) \\
\phi_{1}(-\infty)=0, \,\,\phi_{1}(+\infty)=1 \\
\phi_{2}(-\infty)=0, \,\,\phi_{2}(+\infty)=1. \\
\end{array}
\right.
\end{equation}
Then, we have the following properties. 
\begin{itemize}
\item[(a)] {\em Uniqueness of the velocity:}
Under the additional assumption $(C)$, the velocity $ c $ is unique. 
\item[(b)] {\em Uniqueness of $(\phi_{1}, \phi_{2})$: }
If $ c \neq 0 $, then under the additional assumption $ (C) $ and $ (D+)\,i) $ or $ ii)$ if $c>0$ (resp. $ (D-)\,i) $ or $ ii)$ if $c<0$), $ (\phi_{1}, \phi_{2})$ is unique (up to translation) and $\phi_1$ and $\phi_2$ are increasing.
\end{itemize}
\end{theo}

\begin{rem}
We note that $F(X_0,X_1,X_2)=X_1+X_2-2X_0-\sin(2\pi(X_0-L))-\sin(2\pi L)$ satisfies assumptions (A), (B), (C) , (D+)ii) and (D-)ii). Then Theorem \ref{th:0} is a direct application of Theorems \ref{Th 1} and \ref{th 2bis}.
\end{rem}
\bigskip

For this paper, we define
\begin{equation}\label{eq 6}
r^{*}=\max_{i=0...N}{ \left| r_{i}  \right|} 
\end{equation}
and we assume that $ r^{*}>0 $ (otherwise, the system reduce to a single ODE).
%-----------------------------------------------
\subsection{Organization of the paper}
In Section \ref{sec:2}, we give the definition of viscosity solution and of Hull function. We also recall some basic results about monotone functions. Section \ref{sec:3} is devoted to the proof of existence of traveling waves, namely Theorem \ref{Th 1}.
In Section \ref{sec:4}, we study the question of uniqueness of the velocity by proving a comparison principal on the half line. Finally, in Section \ref{sec:5}, we prove the uniqueness of the profile using different types of strong maximum principles.

%%%%%%%%%%%%%%%%%%%%%%%%%%%%%%%%%%%%%%%%%%
%%%%%%%%%%%%%%%%%%%%%%%%%%%%%%%%%%%%%%%%%%
\section{Preliminary results}\label{sec:2}
%%%%%%%%%%%%%%%%%%%%%%%%%%%%%%%%%%%%%%%%%%
 This section is divided into four subsections. The first one is devoted to the extension of the function $F$ onto $ \mathbb{R}^{N+1} $. In the second subsection,
we give the definition of viscosity solution while the notion of hull functions is recalled in the third one. Finally, we present some results about monotone functions in the last subsection.
\subsection{Extension of F} 
To construct the traveling waves, we will use the hull functions constructed in \cite{FIM2}. To do that, as in \cite{AFM12}, we will need  to extend the function $ F $ by
$\tilde{F}$ which is defined over $ \mathbb{R}^{N+1} $ and satisfied the following assumption:
$\\$
\textbf{ Assumption ($\tilde{A}$):} $ \\ $
 a) \underline{Regularity:}
$\tilde{F} \textrm{ is Lipschitz continuous over } \mathbb{R}^{N+1} .$\newline
b) \underline{Periodicity:} $\tilde{F}(X_{0}+1,...,X_{N}+1)=\tilde{F}(X_{0},...,X_{N})$ for every $X=(X_{0},...,X_{N}) \in \mathbb{R}^{N+1}$.\newline
c) \underline{ Monotonicity:} 
 $ \tilde{F}(X_{0},...,X_{N})$ is non-decreasing in $ V_{i} $ for $ i \neq 0 $ and
\begin{equation}\label{eq:monV0bis}
2\,\dfrac{\partial \tilde{F}}{\partial X_{0}}+ \alpha_0 > 0.
\end{equation}
%---------------------------------
Then, we have the following extension of the function $F$.
 \begin{lem}\label{lem 1} 
Given a function  $F$ defined over  $ Q= [0,1]^{N+1}$ satisfying $(A)$ and $ F(1,...,1)=F(0,...,0)$, there exists an extension
$ \tilde{F}$ defined over 
$\mathbb{R}^{N+1}$ such that
$$ \tilde{F}_{\mid_{Q}}=F \,\,\mbox{and}\,\, \tilde{F}\,\,\mbox{satisfies}\,\, (\tilde{A}). $$ 
\end{lem}
%---------------------------------
\begin{proof}
The construction is made in \cite[Lemma 2.1]{AFM12}. The only thing to verify is that $\tilde F$ satisfy \eqref{eq:monV0bis} if $F$ satisfy \eqref{eq:monV0}, but this is trivial by looking to the way the function $\tilde F$ is constructed.
\end{proof}
%---------------------------------
\begin{rem}\label{rem 1} $\\$
We remark that, if $(\phi_{1}, \phi_{2})$ is a traveling wave for equation \eqref{eq 4} with $F$ replaced by $ \tilde{F} $, then $(\phi_{1}, \phi_{2})$ is also a traveling wave
 of the same equation. This is a direct consequence of Lemma \ref{lem 1} and the fact that
$$\begin{cases}
(\phi_{1}, \phi_{2}) \mbox{\,\,is non-decreasing over}\,\,\,\mathbb{R} \\
\phi_{1}(-\infty)=0, \,\,\phi_{1}(+\infty)=1 \\
\phi_{2}(-\infty)=0, \,\,\phi_{2}(+\infty)=1.\\
\end{cases}$$ 
\end{rem}
%---------------------------------
\noindent Theorem \ref{Th 1} is then a direct application of the following result.
\begin{pro}[Existence of traveling waves]\label{pro 1}
We assume that  $ \tilde F $ satisfies  $ (\tilde{A})$ and $(B)$. Then there exist a real $ c $  and two 
functions $\phi_{1},\phi_{2}$ solutions of 
\begin{equation}\label{eq 7}
\left\lbrace
\begin{array}{lcl}
 c\,\,\phi_{1}'(z) = \alpha_{0}(\phi_{2}(z)-\phi_{1}(z)) \\
c\,\,\phi_{2}'(z)=2\,\tilde F((\phi_{1}(z+r_{i}))_{i=0,...,N})+\alpha_{0}(\phi_{1}(z)-\phi_{2}(z)) \\
\phi_1,\; \phi_2 \textrm{ are non-decreasing over }\R\\
\phi_{1}(-\infty)=0, \phi_{1}(+\infty)=1 \\
\phi_{2}(-\infty)=0, \phi_{2}(+\infty)=1 \\
\end{array}
\right.
\end{equation}
in the classical sense if $ c \neq 0$ and almost everywhere if $ c=0$.
\end{pro}
%---------------------------------
\begin{proof}[Proof of Theorem  \ref{Th 1}]
The proof of Theorem \ref{Th 1} is a direct consequence of Remark \ref{rem 1} and Proposition \ref{pro 1}.
\end{proof}
For simplicity of presentation, we call  $ \tilde{F} $ as $F$ in the rest of this section and in Section \ref{sec:3}.
%%%%%%%%%%%%%%%%%%%%%%%%%%%%%%%%%%%%
\subsection{{Viscosity solution}} 
In this subsection, we give the definition of viscosity solution. We first recall the definition of the upper and the lower semi-continuous envelopes
$u^{*} $  and $ u_{*} $ :
 $$ u^{*}(y)= \underset{ x \to y}{\limsup}\,u (x)$$
$$ u_{*}(y)=\underset{ x \to y}{\liminf} \,u (x).$$ 

%-----------------------------
\begin{defi}[Viscosity solution]\label{def 1}
Let $c \in \mathbb{R}$ and $ F $  be defined over 
 $ \mathbb{R}^{N+1}$. Let $ u_{1}: \mathbb{R} \to  \mathbb{R}$ and $ u_{2 }: \mathbb{R} \to  \mathbb{R}$ be two  locally bounded and upper semi-continuous functions. 
$(u_{1}, u_{2})$ is called a sub-solution on an open set $\Omega \subset \mathbb{R} $ of
\begin{equation}\label{eq 8}
\left\{
\begin{array}{lcl}                                                                                                                                                                                                                                                                                                                                                                                                                                                                                                                                                              
c\,\,\,u_{1}'(z) = \alpha_{0}(u_{2}(z)-u_{1}(z))\\
c \,\,\,u_{2}'(z)=2\,F((u_{1}(z+r_{i}))_{i=0,...,N})+\alpha_{0}(u_{1}(z)-u_{2}(z))
\end{array}
\right.
\end{equation}
if for any test function $ \psi  \in C^{1}(\Omega) $ such that $(u_{1}-\psi)$ (resp $(u_{2}-\psi))$ reaches a local maximum at a point $ z\in \Omega $ then we have
$$\begin{cases}                                                                                                                                                                                                                                                                                                                                                                                                                                                                                                                                                               
c\,\,\,\psi'(z) \le \alpha_{0}(u_{2}(z)-u_{1}(z)) \\
\big( \mbox{resp}.\,\,c \,\,\,\psi'(z) \le 2\,F((u_{1}(z+r_{i}))_{i=0,...,N})+\alpha_{0}(u_{1}(z)-u_{2}(z))\big).\\
\end{cases}$$ \medskip

\noindent Let $ u_{1}: \mathbb{R} \to  \mathbb{R}$ and $ u_{2 }: \mathbb{R} \to  \mathbb{R}$ be two  locally bounded and lower semi-continuous functions. 
$(u_{1}, u_{2})$ is called a super-solution of \eqref{eq 8} on $\Omega$ if for any test function $ \psi  \in C^{1}(\Omega) $ such that $(u_{1}-\psi)$ (resp $(u_{2}-\psi))$ reaches a local minimum at a point $ z\in \Omega $ then we have
$$\begin{cases}                                                                                                                                                                                                                                                                                                                                                                                                                                                                                                                                                               
c\,\,\,\psi'(z) \ge \alpha_{0}(u_{2}(z)-u_{1}(z)) \\
\big( \mbox{resp}.\,\,c \,\,\,\psi'(z) \ge 2\,F((u_{1}(z+r_{i}))_{i=0,...,N})+\alpha_{0}(u_{1}(z)-u_{2}(z))\big).\\
\end{cases}$$ \medskip

Finally, a locally bounded functions $ (u_{1}, u_{2})$ is called a viscosity solution of   (\ref{eq 8}) if  $ ((u_{1})^{*}, (u_{2})^{*})$ is a 
sub-solution and $ ((u_{1})_{*}, (u_{2})_{*}) $ is a super-solution.
\end{defi}

%%%%%%%%%%%%%%%%%%%%%%%%%%%%%
\subsection{Hull  fonction}
We present the notion of hull function for (\ref{eq 3}). This result has been proved in \cite[Theorem 1.10]{FIM2}.
\begin{pro}[Existence of hull functions]\label{defi 2}  
Let F be a given function satisfying ($\tilde{A}$) and let $p > 0 $. Then there exists a unique  $ \lambda_{p} \in \mathbb{R} $ such that there exists two 
locally bounded functions $ h_{p}:\mathbb{R} \to \mathbb{R} $ and  $ g_{p}:\mathbb{R} \to \mathbb{R}$ satisfying (in the viscosity sense): 
\begin{equation}\label{eq 9}
\left\lbrace
\begin{array}{lcl}
\lambda_{p} h_{p}'(x)= \alpha_{0} (g_{p}(x)-h_{p}(x)) \\
\lambda_{p}  g_{p}'(x)= 2 F((h_{p}(x+p\,r_{i}))_{i=0...n} )+ \alpha_{0} (h_{p}(x)-g_{p}(x)) \\
 h_{p}(x+1)=h_{p}(x)+1 \\
 g_{p}(x+1)=g_{p}(x)+1 \\
 h_{p}'(x) \geq 0\\
  g_{p}'(x) \geq 0.
\end{array}
\right.
\end{equation}
\end{pro}
We then define
\begin{equation}\label{eq 11}
\left\lbrace
\begin{array}{lcl}
 \phi_{p}^{1}(x)=h_{p}(p\,x) \\
\phi_{p}^{2}(x)=g_{p}(p\,x) \\
\end{array}
\right.
\qquad{\rm and}\qquad c_{p}=\dfrac{\lambda_{p}}{p}. 
\end{equation}
We now give some properties of the function $(\phi_{1}^{p},\phi_{2}^{p})$.
\begin{lem}[Properties of  $(\phi_{1}^{p},\phi_{2}^{p})$]\label{lem 3} 
We assume that F satisfies ($\tilde{A}$). Then the function  $(\phi_{1}^{p},\phi_{2}^{p})$ defined in (\ref{eq 11}) satisfies in the viscosity sense
\begin{equation}\label{eq 13}
\left\lbrace
\begin{array}{lcl}
c_{p}  (\phi_{p}^{1})'= \alpha_{0} (\phi_{p}^{2}-\phi_{p}^{1}) \\
\\
c_{p} (\phi_{p}^{2})'= 2 F((\phi_{p}^{1}(x+p\,r_{i})_{i=0,...,N}))+ \alpha_{0} (\phi_{p}^{1}-\phi_{p}^{2}) \\
\\
 \phi_{p}^{1}\left(x+\dfrac{1}{p}\right)=\phi_{p}^{1}(x)+1,\;
\phi_{p}^{2}\left(x+\dfrac{1}{p}\right)=\phi_{p}^{2}(x)+1 \\
\\
 (\phi_{p}^{1})'(x) \geq 0,\;
(\phi_{p}^{2})'(x) \geq 0. 
\end{array}
\right.
\end{equation}

Moreover, if $ c_{p} \neq 0$ then there exists $ M >0 $ independent on $p$ such that
\begin{equation}\label{eq 15}
 \left| (\phi_{i}^{p})'  \right| \le \dfrac{M}{ \left| c_{p}  \right|} \quad \mbox{for}\quad 0< p <\dfrac{1}{r^{*}} \quad \mbox{and}\quad i=1,2. 
 \end{equation}
\end{lem}

%--------------------------------------
\begin{proof}[Proof of Lemma \ref{lem 3}]
Equation (\ref{eq 13}) is obtained by the change of
variables  (\ref{eq 11}) in (\ref{eq 9}).

We now prove (\ref{eq 15}). We fix $ p>0 $ such that  
$$ \dfrac{1}{p} \ge r^{*} .$$ 
We first remark that the function $\psi$ defined by $ \psi(x)= \alpha_{0} (\phi_{p}^{2}(x)-\phi_{p}^{1}(x))$ is bounded (because $\psi$ is continuous and periodic) then there exists $ M_{1} >0$ such that
$$ \left|  \psi(x)  \right| \le M_{1}. $$ 
This implies that 
$$\left| (\phi_{1}^{p})'  \right| \le \dfrac{M_1}{ \left| c_{p}  \right|} .$$\medskip

\noindent On the other side, since $\phi^{1}_{p}$ is non-decreasing, we have
$$\begin{cases}
\left|\phi_{p}^{1}(x+r_{i})-\phi_{p}^{1}(x) \right| \le  \left| \phi_{p}^{1}\left(x+\dfrac{1}{p}\right)-\phi_{p}^{1}(x)  \right|=1 \quad \mbox{if}\quad r_{i} \ge 0  \\
\\
\left|\phi_{p}^{1}(x+r_{i})-\phi_{p}^{1}(x) \right| \le  \left|  \phi_{p}^{1}\left(x-\dfrac{1}{p}\right)-\phi_{p}^{1}(x)  \right|=1 \quad \mbox{if}\quad r_{i} \le 0 . \\
\end{cases}$$ 
Moreover, using that $ F \in \Lip(\mathbb{R}^{N+1}) $, we get
$$ \left| F((\phi_{p}^{1}(x+r_{i}))_{i=0,..., N})- F((\phi_{p}^{1}(x))_{i=0,...,N}) \right| \le L \begin{vmatrix} 
1 \\
\vdots\\
1
 \end{vmatrix}  =: L^{1}.$$

\noindent On the other hand, $ f$ is bounded (because
 $f$  is
Lipschitz continuous and periodic). Therefore
$$ \left|   F((\phi_{p}^{1}(x+r_{i})_{i=0,..., N})) \right| \le L^{1}+\left| f\right|_{L^{\infty}(\mathbb{R}) }$$ 
and 
$$  \left| 2 \,F((\phi_{p}^{1}(x))_{i=0,...,N})+ \alpha_{0} (\phi_{p}^{2}(x)-\phi_{p}^{1}(x)) 
 \right| \le 2 (L^{1}+\left| f \right|_{L^{\infty}(\mathbb{R})})+M_{1}=: M_{2}.$$
This implies that
$$ \left|  (\phi_{2}^{p})'   \right| \le \dfrac{M_{2}}{  \left| c_{p}  \right|}. $$ 
Taking $M=\max(M_1,M_2)$, we get the desired result.
\end{proof}
%%%%%%%%%%%%%%%%%%%%%%%%%%%%%
\subsection{Useful results for monotone functions}
In this subsection, we recall some results about monotone function that will be used later for the proof of Proposition \ref{pro 1}. 
We state Helly's Lemma and the equivalence between viscosity and almost everywhere solution. \medskip

\noindent First we recall  Helly's Lemma which gives the convergence of subsequence in the almost everywhere sense.
\begin{lem}[Helly's Lemma]\label{lem 4}
Let  $(g_{n})_{n \in  \mathbb{N}} $  be a sequence of non-decreasing functions on $[a,b]$ verifying $ \left| g_{n} \right| \le C $. Then there exists a subsequence
 $ (g_{n_{j}})_{j\in \N} $ such that 
$$ g_{n_{j}} \to g\quad \mbox{a.e. on}\quad [a,b] $$
where $g$ is non-decreasing on $ [a, b] $ and $ \left| g \right| \le C $.

Moreover, if $(g_{n})_{n\in\N}$ is a sequence of non-decreasing functions on
a bounded interval $I$ and if 
$$g_{n}\rightarrow
g\quad\mbox{a.e. on}\quad I
$$
with $g$ constant on $\mathring{I},$
then for every closed subset interval $I'\subset \mathring{I},$
$$g_{n}\rightarrow g\quad\mbox{uniformly on}\quad I'.$$
 \end{lem}
 %-------------------------------
 \begin{proof}
The first part of this lemma is the classical Helly's Lemma and a proof can be found in \cite[Section 3.3, p. 70]{AGS08} while the second part is proved in \cite[Lemma 2.10]{AFM12}.
 \end{proof}
 
 \noindent Finally, after the use of Lemma \ref{lem 4} we often need to apply the following lemma (which proof is very similar to the one of \cite[Lemma 2.11]{AFM12}) in order to get a solution in the viscosity sense.
\begin{lem}[Equivalence between viscosity and a.e. solutions]\label{lem 6}
We assume that $ F $ satisfies ($\tilde{A}$). Then $\phi_{1}$ and $ \phi_{2} $ are viscosity solutions of 
\begin{equation}\label{eq 17}
\left\lbrace
\begin{array}{lcl}
0 = \alpha_{0}(\phi_{2}(x)-\phi_{1}(x)), \\
0 =2\,\,F((\phi_{1}(x+r_{i}))_{i=0....N})+\alpha_{0} (\phi_{1}(x)-\phi_{2}(x)) 
\end{array}
\right.
\end{equation}
if and only if $ \phi_{1}$ and $\phi_{2} $ are solutions in the almost everywhere sense of the same equation.
 \end{lem}

%%%%%%%%%%%%%%%%%%%%%%%%%%%%%%%%%%%%%%%
%%%%%%%%%%%%%%%%%%%%%%%%%%%%%%%%%%%%%%%
\section{Construction of a traveling wave}\label{sec:3}
%%%%%%%%%%%%%%%%%%%%%%%%%%%%%%%%%%%%%%%
This section is divided into two subsections. In the first one, we control the velocity of propagation and give some properties on the plateau of the profiles.
The second subsection is devoted to the proof of Proposition \ref{pro 1}.
%%%%%%%%%%%%%%%%%%%%%%%
\subsection{{Preliminary results}}
We begin to show that the velocity $c_p$ is uniformly bounded in $p$.
 \begin{lem}[Velocity $ c_{p} $ is bounded]\label{lem 8}
Under the assumption $(\tilde{A})$ and $(B)$, let $ c_{p} $ be the velocity given by (\ref{eq 11}). Then there exists $ C > 0 $ such that 
$$    \left| c_{p} \right| \le C \quad\mbox{ for}\quad 0 < p < \dfrac{1}{r^{*}}   \quad\mbox{ with }\quad r^{*}=\displaystyle \max_{i=0,...,N}  \left| r_{i} \right|. $$
 \end{lem}
\begin{proof}
We consider the functions $ \phi_{p}^{1}$ and $ \phi_{p}^{2} $ given by (\ref{eq 11}) and satisfying (\ref{eq 13}). Let $c_{p}$
be the associated velocity given by (\ref{eq 11}). We assume by contradiction that when $ p \to p_{0} \in [0, \dfrac{1}{r^{*}}]$
$$ \lim_{p \to p_{0}} c_{p}=+\infty$$
(the case $ c_{p} \to -\infty $ being similar). Let
$\bar{\phi}_{p}^{1}=\phi_{p}^{1}(c_{p}\,x)$ and  $\bar{\phi}_{p}^{2}=\phi_{p}^{2}(c_{p}\,x)$ solution of
\begin{equation}\label{eq::1}
\begin{cases}
   (\bar{\phi}_{p}^{1})'= \alpha_{0} (\bar{\phi}_{p}^{2}-\bar{\phi}_{p}^{1}) \\
  (\bar{\phi}_{p}^{2})'= 2\,F((\bar{\phi}_{p}^{1}(x+\dfrac{r_{i}}{c_{p}}))_{i=0,...,N})+\alpha_{0}(\bar{\phi}_{p}^{1}-\bar{\phi}_{p}^{2}).
  \end{cases}
\end{equation}
According to (\ref{eq 15}), we have 
$$\begin{cases}
      \left| (\bar{\phi}_{p}^{1})'\right|=\left| c_{p} (\phi_{p}^{1})' \right| \le M \\
     \left| (\bar{\phi}_{p}^{2})'\right|=\left| c_{p} (\phi_{p}^{2})' \right| \le M
  \end{cases}
$$
for $M$ independent of $ p $. 
Since \eqref{eq::1} is invariant by space translation, we assume that 
$$ \bar{\phi}_{p}^{1}(0)=b-\varepsilon. $$
for $ \varepsilon $ small enough.
Using Ascoli theorem and diagonal extraction argument, we have, up to extract a subsequence, that
$$\begin{cases}
    \bar{\phi}_{p}^{1} \to \bar{\phi}^{1} \\
  \bar{\phi}_{p}^{2} \to \bar{\phi}^{2}. 
  \end{cases}
$$ 
Moreover, by stability of viscosity solutions, $\bar{\phi}^{1}\,\,\mbox{and}\,\,\bar{\phi}^{2} $ satisfy
$$ \begin{cases}
(\bar{\phi}^{1})'(x)= \alpha_{0} (\bar{\phi}^{2}(x)-\bar{\phi}^{1}(x)) \\
  (\bar{\phi}^{2})'(x)= 2\,F((\bar{\phi}^{1}(x))_{i=0,...,N})+\alpha_{0}(\bar{\phi}^{1}(x)-\bar{\phi}^{2}(x))
\end{cases}
$$
and
$$ \bar \phi^{1}(0)=b-\varepsilon. $$
Since $ (\bar{\phi}^{1}_{p})' \ge 0, (\bar{\phi}^{2}_{p})' \ge 0 $, we have $ (\phi^{1})'\ge 0, (\phi^{2})' \ge 0 $.
This implies that 
$$\begin{cases}
   \alpha_{0}(\bar{\phi}^{2}(x)-\bar{\phi}^{1}(x)) \ge 0 \\
   2\,f(\bar{\phi}^{1}(x))+\alpha_{0}(\bar{\phi}^{1}(x)-\bar{\phi}^{2}(x)) \ge 0. \\
  \end{cases}
$$
 Therefore
 $$ 2\,f(\bar{\phi}_{1}(x)) \ge 0. $$
 In particular, $ 2\,f(\bar{\phi}^{1}(0))= 2\,\,f(b-\varepsilon)\ge 0$, which is a contradiction since
 $ f(b-\varepsilon) < 0$ (see assumption $ (B)$).
 
\end{proof}

%----------------------------------
\noindent We continue with some properties on the plateau of the profiles. The following lemma shows that if one of the profile have a large enough plateau then the other profile has the same plateau.
\begin{lem}[Properties on the plateau of the profiles]\label{lem:2}
Let $(\phi_{1}, \phi_{2})$ be solution of 
\begin{equation}\label{eq:03}\begin{cases}
  c \phi_{1}'(x)=\alpha_{0} (\phi_{2}(x)-\phi_{1}(x)) \\
  c \phi_{2}'(x)=2\,\,F((\phi_{1}(x+r_{i}))_{i=0,...,N})+\alpha_{0} (\phi_{1}(x)-\phi_{2}(x))\\
  \phi_1'\ge0,\; \phi_2'\ge 0.
 \end{cases}
\end{equation}
We assume that there exists a constant $C$, a point $x_0\in \R$ and $a>r^*$ such that
$$\phi_1(x)=C\quad\forall x\in (x_0-a,x_0+a)\quad{\rm or}\quad\phi_2(x)=C\quad \forall x\in (x_0-a,x_0+a).$$
Then 
$$\phi_1(x)=\phi_2(x)=C\quad\forall x\in (x_0-a,x_0+a).$$
\end{lem}
%---------------------------------
\begin{proof}
If $\phi_1(x)=C\quad \forall x\in (x_0-a,x_0+a)$, then $\phi_1'(x)=0\,\,\;\forall x\in(x_0-a,x_0+a)$ and the first equation of \eqref{eq:03} implies the result.

Let us then assume that
$$\phi_2(x)=C\quad \forall x\in (x_0-a,x_0+a).$$
We set 
$$\psi_1(x)=(\phi_1)_*(x+a)-(\phi_1)^*(x-a)\quad {\rm and}\quad\psi_2(x)=(\phi_2)_*(x+a)-(\phi_2)^*(x-a).$$
Then $(\psi_1,\psi_2)$ is solution of
\begin{equation}\label{eq:04}
\left\{\begin{array}{rcl}
c\psi_1'(x)&\ge& \alpha_0(\psi_2(x)-\psi_1(x))\\
c\psi_2'(x)&\ge& 2\left[F(((\phi_1)_*(x+a+r_i))_{i=0,\dots,N})-F(((\phi_1)^*(x-a+r_i))_{i=0,\dots,N})\right]\\
&&+\alpha_0(\psi_1(x)-\psi_2(x)).
\end{array}
\right.
\end{equation}
Since $\phi_1$ and $\phi_2$ are non-decreasing, we have $\psi_1\ge 0$ and $\psi_2\ge 0$. Moreover, $\psi_2(x_0)=0$. Hence, $x_0$ is a point of minimum of $\psi_2$ and the second equation of \eqref{eq:04} implies that
\begin{align*}
0\ge&2\left[F(((\phi_1)_*(x_0+a+r_i))_{i=0,\dots,N})-F(((\phi_1)^*(x_0-a+r_i))_{i=0,\dots,N})\right]+\alpha_0\psi_1(x_0)\\
\ge&2\left[F(( (\phi_1)^*(x_0-a)+\psi_1(x_0),  (\phi_1)^*(x_0-a+r_i))_{i=1,\dots,N})-F(((\phi_1)^*(x_0-a+r_i))_{i=0,\dots,N})\right]\\
+&\alpha_0\psi_1(x_0)
\end{align*}
where we have used the monotony of $F$ for the second inequality. We set
$$G(x)=2F(( (\phi_1)^*(x_0-a)+x,  (\phi_1)^*(x_0-a+r_i))_{i=1,\dots,N})+\alpha_0 x.$$
Then, by assumption \eqref{eq:monV0}, $G$ is strictly increasing. Using that
$$0\ge G(\psi_1(x_0))-G(0),$$
we deduce that $\psi_1(x_0)=0$ (recall that $\psi_1\ge 0$). This implies that $\phi_1$ is constant over $(x_0-a,x_0+a)$ and by the first equation of \eqref{eq:03}, this constant is also equal to $C$.

\end{proof}
%---------------------------------
In the proof of Proposition \ref{pro 1}, we will need to pas to the limit for $(\phi^1_p,\phi^2_p)$. This is the goal of the following lemma.
\begin{lem}[Passing to the limit for $(\phi^1_p,\phi^2_p)$]\label{lem:1}
For every $n\in \N$, let $(c^n,\phi_1^n,\phi_2^n)$ be a solution of
\begin{equation}\label{eq:05}\begin{cases}
  c^n (\phi_{1}^n)'(x)=\alpha_{0} (\phi_{2}^n(x)-\phi_{1}^n(x)) \\
  c (\phi_{2}^n)'(x)=2\,\,F((\phi_{1}^n(x+r_{i}))_{i=0,...,N})+\alpha_{0} (\phi_{1}^n(x)-\phi_{2}^n(x))\\
  (\phi_1^n)'\ge0,\; (\phi_2^n)'\ge 0
 \end{cases}
\end{equation}
satisfying 
$$\left\{\begin{array}{l}
\phi_1^n(x+1)\le \phi_1^n(x)+1\\
\phi_2^n(x+1)\le \phi_2^n(x)+1\\
|c^n|\le M_0\\
|c^n(\phi_1^n)'|\le M_1,\; |c^n(\phi_2^n)'|\le M_2
\end{array}\right.$$
where $M_0, M_1$ and $M_2 $ are positive constant.
We also assume that there exists $M_3>0$ and $x^*\in \R$ such that $|\phi_1^n(x^*)|\le M_3$. 

Then there exists $(c,\phi_1, \phi_2)$ such that, up to extract a subsequence, 
$c^n\to c$, $\phi_1^n\to\phi_1$ and $\phi_2^n\to\phi_2$ a.e. and $(c,\phi_1,\phi_2)$ is a viscosity solution of
\begin{equation}\label{eq:06}\begin{cases}
  c \phi_{1}'(x)=\alpha_{0} (\phi_{2}(x)-\phi_{1}(x)) \\
  c \phi_{2}'(x)=2\,\,F((\phi_{1}(x+r_{i}))_{i=0,...,N})+\alpha_{0} (\phi_{1}(x)-\phi_{2}(x))\\
  \phi_1'\ge0,\; \phi_2'\ge 0
 \end{cases}
\end{equation}
\end{lem}
%---------------------------------
\begin{proof}
Up to translate $\phi_1$, we assume that $x^*=0$. Since $\left|  c^{n} \right|  \le M_{0}$, up to extract a subsequence, we can assume that
$$ c^n\to c\,\, \quad \mbox{as} \quad \,n \to +\infty . $$
We study two cases for $c$.

\paragraph{Case 1: $ c \neq 0 $.}
For $n$ large enough, we have $  \left| c^n \right| \ge \dfrac{\left| c \right|}{2} \neq 0$. Hence for $n$ large enough, we have
$$ \left| (\phi_{1}^{n})' \right| \le \dfrac{2 \,M_{1}}{\left| c  \right|} \quad \mbox{and }\quad \left|(\phi_{2}^{n})'\right| \le \dfrac{2 \,M_{2}}{\left| c  \right|}.$$
Using Ascoli's Theorem and the diagonal extraction argument, we can assume, up to a subsequence, that $ (\phi_{1}^{n})_{n}$ and $(\phi_{2}^{n})_{n}$ converge locally uniformly on $ \mathbb{R}$ respectively to $\phi_{1} $ and $\phi_{2} $. By stability, $\phi_{1}$ and $ \phi_{2} $ satisfy in the viscosity sense 
\begin{equation}\label{eq 27}
\left\lbrace
\begin{array}{lcl}
 c \phi_{1}'(x) = \alpha_{0}(\phi_{2}(x)-\phi_{1}(x)) \\
c \phi_{2}'(x)=2\,\,F((\phi_{1}(x+r_{i}))_{i=0,...,N})+\alpha_{0} (\phi_{1}(x)-\phi_{2}(x)) \\
\phi_{1}' \ge 0, \; \phi_{2}' \ge 0.
\end{array}
\right.
\end{equation}

\paragraph{Case 2: $ c=0$.}
 We have $ \phi_{1}^{n}(1+x) \le \phi_{1}^{n}(x)+1$. This implies, using the fact that $\phi_1^n(0)\le M_3$, that
 \begin{equation}\label{eq 28}
\left\lbrace
\begin{array}{lcl}
 \phi_{1}^{n}(x) \le \lceil x \rceil + M_3 \quad \mbox{ for}\quad x \ge 0 \\
 \phi_{1}^{n}(x) \ge - \lceil \left| x \right| \rceil -M_3 \quad \mbox{ for}\quad x \le 0
 \end{array}
\right.
\end{equation}
 Using also the fact that $0\le| c^n (\phi^n_1)'|\le M_1$, we get
\begin{equation}\label{eq 28bis}
\left\lbrace
\begin{array}{lcl}
 \phi_{2}^{n}(x) \le \lceil x \rceil + \dfrac{M_{1}}{\alpha_{0}}+M_3 \quad \mbox{ for}\quad x \ge 0 \\
 \phi_{2}^{n}(x) \ge - \lceil \left| x \right| \rceil  -M_3\quad \mbox{ for}\quad  x \le 0
\end{array}
\right.
\end{equation}
Using Helly's Lemma (Lemma \ref{lem 4}), up to extract  a subsequence, we have
 $\phi_{1}^{n} \to \phi_{1}$ a.e. and  $\phi_{2}^{n}\to\phi_{2}$ a.e.. 
This implies that
$$\begin{cases}
 c^{n} \int_{b_{1}}^{b_{2}}  (\phi_{1}^{n})'(x) dx = \alpha_{0} \int_{b_{1}}^{b_{2}}(\phi_{2}^{n}(x)-\phi_{1}^{n}(x)) dx \\
c^{n} \int_{b_{1}}^{b_{2}} (\phi_{2}^{n})'(x) dx =2\,\, \int_{b_{1}}^{b_{2}} F((\phi_{1}^{n}(x+r_{i}))_{i=0,...,N}) dx +\alpha_{0} \int_{b_{1}}^{b_{2}}
(\phi_{1}^{n}(x)-\phi_{2}^{n}(x)) dx .\\
\end{cases}$$
for every $b_{1} \le b_{2}$. That is 
$$\begin{cases}
  c^{n} (\phi_{1}^{n}(b_{2})-\phi_{1}^{n}(b_{1})) = \alpha_{0} \int_{b_{1}}^{b_{2}}(\phi_{2}^{n}(x)-\phi_{1}^{n}(x)) dx \\
 c^{n} (\phi_{2}^{n}(b_{2})-\phi_{2}^{n}(b_{1})=2\,\, \int_{b_{1}}^{b_{2}} F((\phi_{1}^{n}(x+r_{i}))_{i=0,...,N}) dx +\alpha_{0} \int_{b_{1}}^{b_{2}}
(\phi_{1}^{n}(x)-\phi_{2}^{n}(x)) dx. \\
\end{cases}$$
We have 
$$\begin{cases}
 \alpha_{0} (\phi_{2}^{n}(x)-\phi_{1}^{n}(x)) \to \alpha_{0} (\phi_{2}(x)-\phi_{1}(x))\,\,\mbox{a.e.}\, \\
2\,\,\,F((\phi_{1}^{n}(x+r_{i}))_{i=0,...,N})+ \alpha_{0} (\phi_{1}^{n}(x)-\phi_{2}^{n}(x))\,\,\to\,\,2\,\,F((\phi_{1}(x+r_{i}))_{i=0,...,N})+\\
\qquad\qquad\qquad\qquad\qquad\qquad\qquad\qquad\qquad\qquad\qquad\qquad\alpha_{0}(\phi_{1}(x)-\phi_{2}(x))\,\,\mbox{a.e.}
\end{cases}$$
and (because of  (\ref{eq 28}) and $ F $ is Lipschitz continuous and $ f$ is bounded) 
\begin{align*}
&\left| \alpha_{0} (\phi_{2}^{n}(x)-\phi_{1}^{n}(x))  \right| \le M_{1} \\
&\left| F((\phi_{1}^{n}(x+r_{i}))_{i=0,...,N}) \right |\le C (1 +\left| x\right|),
\end{align*}
i.e.
$$\left| F((\phi_{1}^{n}(x+r_{i}))_{i=0,...,N}) + \alpha_{0} (\phi_{1}^{n}(x)-\phi_{2}^{n}(x)) \right| \le C (1 +\left| x \right|) +  M_{1}.$$
Thus, using Lebesgue's dominated convergence Theorem, we pass to the limit
 as $ n \to +\infty $ and we get
$$\begin{cases}
 0 = \alpha_{0} \int_{b_{1}}^{b_{2}}(\phi_{2}(x)-\phi_{1}(x)) dx \\
 0 =2\,\, \int_{b_{1}}^{b_{2}} F((\phi_{1}(x+r_{i}))_{i=0,...,N}) dx + \alpha_{0} \int_{b_{1}}^{b_{2}}(\phi_{1}(x)-\phi_{2}(x)) dx.
\end{cases} $$
which implies that
$$\begin{cases}
 0 = \alpha_{0} (\phi_{2}(x)-\phi_{1}(x)) \,\,\mbox{a.e.}\\
 0 =2 F((\phi_{1}(x+r_{i}))_{i=0,...,N})+\alpha_{0}\,\,(\phi_{1}(x)-\phi_{2}(x))\,\,\,\,\mbox{a.e.}
\end{cases}$$
Since $ (\phi_{1}^{n})' \ge 0,\,\,(\phi_{2}^{n})' \ge 0 $, we get  $ \phi_{1}' \ge 0,\,\,\phi_{2}' \ge 0 $ and so Lemma \ref{lem 6} implies
\begin{equation}\label{eq 29}
\left\lbrace
\begin{array}{lcl}
 0 = \alpha_{0} (\phi_{2}(x)-\phi_{1}(x)) \\
  0 =2 F((\phi_{1}(x+r_{i}))_{i=0,...,N})+ \alpha_{0}\,\,(\phi_{1}(x)-\phi_{2}(x))\\
  \phi_{1}' \ge 0\\
  \phi_{2}' \ge 0 
\end{array}
\right.
\end{equation}
in the viscosity sense.

\end{proof}
%------------------------------
We finish this subsection with the following proposition which help us to identify the value  of the plateau of the profiles.

\begin{pro}[The value of the plateau of the profile are close to the zero of $f$]\label{pro 3}
We assume that $ F $ satisfies  ($\tilde{A}$) and let $ a > r^{*}$. For every
$ \varepsilon > 0$,  there exists $ \delta= 
\delta(\varepsilon) $ such that for all function
 $ (c,\phi_{1}, \phi_{2}) $ solution of 
$$\begin{cases}
 c \phi_{1}'(x) = \alpha_{0}(\phi_{2}(x)-\phi_{1}(x)) \\
c \phi_{2}'(x)=2\,\,F((\phi_{1}(x+r_{i}))_{i=0,...,N})+\alpha_{0} (\phi_{1}(x)-\phi_{2}(x)) \\
\phi_{1}' \ge 0,\,\phi_{2}' \ge 0 \\
\phi_{1}(x+1) \le \phi_{1}(x)+1,\,\,\phi_{2}(x+1) \le \phi_{2}(x)+1 \\
\left| c \right|  \le M_{0}\\
 \left| c \phi_{1}'  \right|  \le M_{1},\,\,\left| c \phi_{2}'  \right| \le M_{2} \\
\end{cases}$$
and for all $ x_{0} \in \mathbb{R} $ satisfying
$$(\phi_{1})_{*}(x_{0}+a)-(\phi_{1})^{*} (x_{0}-a ) \le \delta (\varepsilon)\quad {\rm or} \quad  (\phi_{2})_{*}(x_{0}+a)-(\phi_{2})^{*} (x_{0}-a ) \le \delta (\varepsilon),$$
we have  
$$dist(\alpha_{1},\{0,b\}+\mathbb{Z})< \varepsilon \quad \mbox{for\,\,all}\quad \alpha_{1} \in [(\phi_{1})_{*}( x_{0}), (\phi_{1})^{*}(x_{0})]$$
and
$$dist(\alpha_{2},\{0,b\}+\mathbb{Z})< \varepsilon \quad \mbox{for\,\,all}\quad \alpha_{2} \in [(\phi_{2})_{*} ( x_{0}), (\phi_{2})^{*}(x_{0})]. $$ 
\end{pro}
%--------------------------

\begin{proof}
The proof is decomposed into three steps. 

\paragraph{Step 1: Construction of a sequence.} 
We assume by contradiction that there exists $ \varepsilon >0 $ such that for all $ \delta_{n} \to 0 $ and
$ (c^n,\phi_{1}^{n},\phi_{2}^{n})$ solution of 
\begin{equation}\label{eq 22}
\left\lbrace
\begin{array}{lcl}
 c^{n} (\phi_{1}^{n})'(x) = \alpha_{0}(\phi_{2}^{n}(x)-\phi_{1}^{n}(x)) \\
c^{n} (\phi_{2}^{n})'(x)=2\,\,F((\phi_{1}^{n}(x+r_{i}))_{i=0,...,N})+\alpha_{0} (\phi_{1}^{n}(x)-\phi_{2}^{n}(x)) \\
(\phi_{1}^{n})' \ge 0,\,\,(\phi_{2}^{n})' \ge 0 \\
\phi_{1}^{n}(x+1) \le \phi_{1}^{n}(x)+1,\,\phi_{2}^{n}(x+1) \le \phi_{2}^{n}(x)+1 \\
\left| c^{n} \right|  \le M_{0}\\
 \left| c^{n} (\phi_{1}^{n})'  \right|  \le M_{1},\,\left| c^{n} (\phi_{2}^{n})'  \right| \le M_{2} 
\end{array}
\right.
\end{equation}
and there exists   $ (x_{n})_{n} \subset \mathbb{R} $ satisfying 
\begin{equation}\label{eq 23}
(\phi_{1}^{n})_{*}(x_{n}+a)-(\phi_{1}^{n})^{*} (x_{n}-a) \le \delta_{n} \to 0 \quad{\rm or}\quad (\phi_{2}^{n})_{*}(x_{n}+a)-(\phi_{2}^{n})^{*} (x_{n}-a) \le \delta_{n} \to 0
\end{equation}
and $ \alpha_{1}^{n} \in [(\phi_{1}^{n})_{*}(x_{n}), (\phi_{1}^{n})^{*}(x_{n})]$ such that
\begin{equation}\label{eq 24}
  dist\,(\alpha_{1}^{n},\{0,b\}+\mathbb{Z}) \ge  \varepsilon > 0
\end{equation}
or $ \alpha_{2}^{n} \in [(\phi_{2}^{n})_{*}(x_{n}), (\phi_{2}^{n})^{*}(x_{n})]$ such that
\begin{equation}\label{eq 25}
dist\,(\alpha_{2}^{n},\{0,b\}+\mathbb{Z}) \ge  \varepsilon > 0.
\end{equation}\medskip

\noindent Up to translate the profile, we assume that
\begin{equation}\label{eq 26}
\left\lbrace
\begin{array}{lcl}
 x_{n} \equiv 0 \\
\phi_{1}^{n}(0) \in [0,1) \quad \mbox{for\,\,all}\quad n . \\
\end{array}
\right.
\end{equation}

\paragraph{Step 2: Passing to limit $ n \to +\infty $.}
Using Lemma \ref{lem:1}, we deduce that there exists $(c,\phi_1,\phi_2)$
 such that, up to extract a subsequence, 
$c^n\to c$, $\phi_1^n\to\phi_1$ and $\phi_2^n\to\phi_2$ a.e. and $(c,\phi_1,\phi_2)$ is solution of
\begin{equation}\label{eq:07}\begin{cases}
  c \phi_{1}'(x)=\alpha_{0} (\phi_{2}(x)-\phi_{1}(x)) \\
  c \phi_{2}'(x)=2\,\,F((\phi_{1}(x+r_{i}))_{i=0,...,N})+\alpha_{0} (\phi_{1}(x)-\phi_{2}(x))\\
  \phi_1'\ge0,\; \phi_2'\ge 0
 \end{cases}
\end{equation}

\paragraph{Step 3: Getting a contradiction.} 
We pass to the limit in (\ref{eq 23}) with $ x_{n}=0 $. This implies that
\begin{equation}\label{eq 30}
(\phi_{1})_{*}(a) \le (\phi_{1})^{*} (-a)\quad {\rm or}\quad (\phi_{2})_{*}(a) \le (\phi_{2})^{*} (-a).
\end{equation}
Since $ \phi_{1}$ and $\phi_2$ are non-decreasing, we have $ \phi_{1}=K_{1}\,\mbox{on}\,[-a, a]$ or $ \phi_{2}=K_{1}\,\mbox{on}\,[-a, a]$.  Using Lemma \ref{lem:2} we then get that
$$\phi_{2}=\phi_{1}=K_{1}\quad\mbox{on}\quad[-a, a].$$
Using \eqref{eq:07}, we deduce that  for $ x=0 $ 
$$ 0=F((\phi_{1}(x+r_{i}))_{i=0,...,N}) =F((K_{1})_{i=0,...,N })=f(K_{1}).$$
Hence $ K_{1} \in \{0,b\}+ \mathbb{Z}$. Although, since $ \alpha_{1}^{n} \in [(\phi_{1}^{n})_{*}(0), (\phi_{1}^{n})^{*}(0)]$ then
 $ \alpha_{1}^{n} \to \alpha_{1} \in \{K_{1}\} .$
But, passing to the limit in (\ref{eq 24}), yields
$$
  dist\,(\alpha_{1},\{0,b\}+\mathbb{Z}) \ge  \varepsilon  > 0,
$$
which is a contradiction. 
Similarly, if we pass to the limit in (\ref{eq 25}), we then get 
$$
  dist\,(\alpha_{2},\{0,b\}+\mathbb{Z}) \ge  \varepsilon > 0,
$$
which is also a contradiction. 

\end{proof}

%%%%%%%%%%%%%%%%%%%%%%%%%%%%%%%%%%%%%%%%%%
\subsection{Proof of Proposition \ref{pro 1}}
We are now able to give the proof of Proposition \ref{pro 1}.
\begin{proof}[Proof of Proposition \ref{pro 1}]~
\paragraph{Step 0: Introduction.}
Let $p>0$ and $(\phi_{p}^{1},\phi_{p}^{2})$  (given by (\ref{eq 11}))  be two non-decreasing functions  solution of 
\begin{equation}\label{eq:01}
\begin{cases}
  c_{p}\,(\phi_{p}^{1})^{'}(x) = \alpha_{0}(\phi_{p}^{2}(x)-\phi_{p}^{1}(x))  \\
c_{p}\,\,(\phi_{p}^{2})^{'}(x)=2\,F((\phi_{p}^{1}(x+r_{i}))_{i=0,...,N})+\alpha_{0}(\phi_{p}^{1}(x)-\phi_{p}^{2}(x)).
\end{cases}
\end{equation}
with 
\begin{equation}\label{eq:08}
 \phi_{p}^{1}\left(x+\dfrac{1}{p}\right)=1+\phi_{p}^{1}(x)\qquad{\rm and}\qquad
\phi_{p}^{2}\left(x+\dfrac{1}{p}\right)=1+\phi_{p}^{2}(x).
\end{equation}
Up to translate $\phi_{p}^{1} $, we assume that 
\begin{equation}\label{eq 31}
\left\lbrace
\begin{array}{lcl}
 (\phi_{p}^{1})_{*}(0) \le b \\
(\phi_{p}^{1})^{*}(0) \ge b .
\end{array}
\right.
\end{equation}
Our goal is to pass to the limit as $ p $ tends to zero. 

\paragraph{Step 1: Passing to the limit $ p \to 0$.}
We want to apply Lemma \ref{lem:1}. The only thing we have to show is that $\phi_p^1(0)$ is bounded.
From (\ref{eq:08}) and (\ref{eq 31}), we deduce that 
$$ b-1 \le \phi_{p}^{1}\left(-\dfrac{1}{2 p}\right ) \le (\phi_{p}^{1})_{*}(0) \le b \le (\phi_{p}^{1})^{*}( 0)
\le \phi_{p}^{1}\left(\dfrac{1}{2 p }\right) \le b+1. $$
Thus 
$$
b-1 \le  \phi_{p}^{1}(0)\le  b+1. \\
$$
Using  Lemma \ref{lem:1}, we then deduce that there exists $(c,\phi_1,\phi_2)$
 such that, up to extract a subsequence, 
$c_p\to c$, $\phi_p^1\to\phi_1$ and $\phi_p^2\to\phi_2$ a.e. and $(c,\phi_1,\phi_2)$ is solution of
\begin{equation}\label{eq 37}\begin{cases}
  c \phi_{1}'(x)=\alpha_{0} (\phi_{2}(x)-\phi_{1}(x)) \\
  c \phi_{2}'(x)=2\,\,F((\phi_{1}(x+r_{i}))_{i=0,...,N})+\alpha_{0} (\phi_{1}(x)-\phi_{2}(x))\\
  \phi_1'\ge0,\; \phi_2'\ge 0.
 \end{cases}
\end{equation}
We also note that $(\phi_1)_*(0)\le b$ and $(\phi_1)^*(0)\ge b$.

\paragraph{Step 2 : Properties of the limit  $(\phi^{1}, \phi^{2})$.}
\paragraph{Step 2.1: The oscillation of  $(\phi^{1}, \phi^{2})$ is bounded.}
Let $R>0$. For every  $p$ such that $ R \le \dfrac{1}{2 p}$ we have
$$ \phi^1_{p}( R)- \phi^1_{p}( -R) \le \phi^1_{p}\left(\dfrac{1}{2 p}\right)- \phi^1_{p}\left( \dfrac{-1}{2 p}\right)=1.$$
Passing to the limit as $p\to 0$, we get
$$ \phi^{1}( R)- \phi^{1}(-R) \le 1.$$
Sending $R\to +\infty$, we deduce that
$$ \phi^{1}( +\infty)- \phi^{1}(-\infty) \le 1.$$
We get in the same way that
$$ \phi^{2}( +\infty)- \phi^{2}( -\infty) \le 1. $$

\paragraph{Step 2.2: $ \phi^{1}(\pm \infty)\in \mathbb{Z} \cup ({b} \pm \mathbb{Z})\,\, and \,\,\phi^{2}(\pm \infty)=\phi^{1}(\pm \infty)$.}
We define
$$\phi^{1}_{n}(x)=\phi^{1}(x-n), \quad {\rm and} \quad \phi^{2}_{n}(x)=\phi^{2}(x-n) .$$
Since (\ref{eq 37}) is invariant by translation, we get that $(\phi^1_n,\phi^2_n)$ is still solution of \eqref{eq 37}. Moreover, since $(\phi^{1},\phi^{2})$ is non-decreasing and bounded,  $(\phi^{1}_{n},\phi^{2}_{n}) $ is also non-decreasing and bounded. Thus
$ (\phi^{1}_{n},\phi^{2}_{n})$ converges as $n\to +\infty $ and we denote by $(\phi^1(-\infty),\phi^2(-\infty))$ its limit. By stability of viscosity solution, we then get that
$$\begin{cases}
 0 = \alpha_{0}(\phi^{2}(-\infty)-\phi^{1}(-\infty)) \\
 0 = 2\, F(((\phi^{1}(-\infty))_{i=0,...,N})+ \alpha_{0}(\phi^{1}(-\infty)-\phi^{2}(-\infty)) .\\
\end{cases}$$
The first equation implies that $\phi^1(-\infty)=\phi^2(-\infty)$ while the second implies that $ f(\phi_{1}(-\infty))=0$ and so $\phi_{1}(-\infty) \in  \mathbb{Z} \cup ({b}+\mathbb{Z})$.

\noindent In the same way, we get  $ \phi_{1}(+\infty) \in  \mathbb{Z} \cup ({b}+\mathbb{Z})$ and $\phi^{1}(+\infty)=\phi^{2}(+\infty) $.

\paragraph{Step 3: $ \phi^{1}(\pm \infty) \notin \{b\} + \mathbb{Z} $.}
Since
$$ \phi^{1}(+\infty)- \phi^{1}(-\infty) \le  1 \quad{\rm and}\quad \begin{cases}
  (\phi^{1})_{*}(0)\le b \\
(\phi^{1})^{*}(0) \ge b,
\end{cases}$$
we obtain that $ \phi^{1}(-\infty) \in \{b-1,0,b\} $ and  $ \phi^{1}(+\infty) \in \{b,1,b+1\} $. If $\phi^{1}(+\infty)=b+1$ then $ \phi^{1}(-\infty)=b $ and if
$\phi^{1}(-\infty)=b-1 $ then $ \phi^{1}(+\infty)= b $. Thus, it is sufficient to exclude the cases $ \phi^{1}(\pm \infty)=b $. At the end, this will prove that
 that $\phi^{1}(+\infty)=1 $ and $\phi^{1}(-\infty)=0  $ (and so by step 2.2, $\phi^{2}(+\infty)=1 $ and $\phi^{2}(-\infty)=0$).
 
 \noindent By contradiction, we assume that
$$ \phi^{1}(+\infty)=b. $$ (the case $\phi^{1}(-\infty)=b$ being similar). Let $ x_{0}= 2 \,r^{*} $, where $ r^{*}=\displaystyle \max_{i=0,...,N}  \left| r_{i} \right| $. 
Since 
$$b=\phi^{1}(+\infty) \ge (\phi^{1})^{*}(0) \ge  b $$
then $\phi^{1}(x)=b$  for all $ x>0 $.
Hence 
$$     \phi^{1}(x_{0})=\phi^{1}( x_{0} \pm a )=b 
$$
for $ r^{*} < a < \,2\,r^{*}  $. 

\paragraph{Step 3.1: Introduce $ z_{p} $ and $ y_{p} $.}
For any $ \varepsilon > 0  $  small enough  ($\varepsilon <\min (b,1-b)/2$), let $ z_{p}\,\,\mbox{and}\,\, y_{p} \in \mathbb{R}$ such that 
\begin{equation}\label{eq 32}
\left\lbrace
\begin{array}{lcl}
\vs(\phi_{p}^{1})_{*}(z_{p}) \le  b+ \varepsilon \\
(\phi_{p}^{1})^{*}(z_{p}) \ge  b+  \varepsilon 
\end{array}
\right.
\quad {\rm and } \quad
\left\lbrace
\begin{array}{lcl}
\vs(\phi_{p}^{1})_{*}(y_{p}) \le b - \varepsilon \\
(\phi_{p}^{1})^{*}(y_{p}) \ge b-  \varepsilon .
\end{array}
\right.
\end{equation}
Let
 $$ \psi_{p}^{1}(x)=(\phi_{p}^{1})_{*} (x+a)-(\phi_{p}^{1})^{*}(x-a )$$ 
$$ \psi_{p}^{2}(x)=(\phi_{p}^{2})_{*}(x+a )-(\phi_{p}^{2})^{*}(x-a ). $$
Note that  $(\psi_{p}^{1},\psi_{p}^{2})$  is lower semi-continuous and solution of
\begin{equation}\label{eq 36}
\left\lbrace
\begin{array}{rcl}
  c_{p} (\psi_{p}^{1})^{'}(x) &\ge& \alpha_{0}(\psi_{p}^{2}(x)- \psi_{p}^{1}(x)) \\
  c_{p} (\psi_{p}^{2})^{'}(x) &\ge&  2\,( F(((\phi_{p}^{1})_{*}(x+a +r_{i}))_{i=0,...,N})-F(((\phi_{p}^{1})^{*}(x-a+r_{i}))_{i=0,...,N})) \\
&&+\alpha_{0}(\psi_{p}^{1}(x)-\psi_{p}^{2}(x)).
 \end{array}
\right.
\end{equation}
Moreover, we have 
 $$ b+\varepsilon \in [(\phi_{p}^{1})_{*}(z_{p}), (\phi_{p}^{1})^{*}(z_{p})] \quad \mbox{ such that} \quad \mbox{dist}\,\,(b+\varepsilon ,\{0, b\}+\mathbb{Z}) \ge \varepsilon .$$ 
Then there exists $ \delta (\varepsilon ) $ (given by Proposition \ref{pro 3}) independent of $ p $ such that (for $ a > r^{*} $)
\begin{equation}\label{eq 34}
 \psi_1(z_p) \ge \delta (\varepsilon ) > 0\quad {\rm and}\quad  \psi_2(z_p) \ge \delta (\varepsilon ) > 0.
\end{equation}
Similarly, we get that 
\begin{equation}\label{eq 35}
 \psi_1(y_p) \ge \delta (\varepsilon ) > 0\quad{\rm and}\quad   \psi_2(y_p) \ge \delta (\varepsilon ) > 0.
\end{equation}

\noindent Using the uniform convergence of $ \phi_{p}^{1}\,\,\mbox{to}\,\,\phi^{1}  $ (see the second part of Lemma \ref{lem 4} if $ c=0 $), we also get that 
$$ \phi_{p}^{1}(x_{0}) \to b $$
and 
 $$
     \psi_{p}^{1}(x_{0})=(\phi_{p}^{1})_{*}(x_{0}+a)-(\phi_{p}^{1})^{*}(x_{0}-a) \to 0 \quad \mbox{as} \quad p\to 0.
     $$

\paragraph{Step 3.2: Equation satisfied by $(\psi^{1}, \psi^{2})$ at its point of minimum.}
Since 
$$\begin{cases}
z_{p} \to +\infty \,\,\mbox{as}\,\, p \to +\infty \\
y_{p}\le 0, 
\end{cases}$$
we have $ x_{0} \in [y_{p}, z_{p}] $ for $ p $ small enough.
We define	
\begin{align*}
&m^{1}_{p}=\min_{x \in [y_{p}, z_{p}]} \psi_{p}^{1}(x)=  \psi_{p}^{1}(x_{p}^{1}) \ge 0 \quad \mbox{with} \quad x_{p}^{1} \in [y_{p}, z_{p}]\\
&m^{2}_{p}=\min_{x \in  [y_{p}, z_{p}]} \psi_{p}^{2}(x)=  \psi_{p}^{2}(x_{p}^{2}) \ge  0 \quad \mbox{with} \quad  x_{p}^{2} \in [y_{p}, z_{p}].
\end{align*}
Note that
\begin{equation}\label{eq 40}
m^{1}_{p}=\psi_{p}^{1} (x_{p}^{1}) \le \psi_{p}^{1}(x_{0}) \to 0 \quad \mbox{as} \quad p \to 0.\\
\end{equation}
Since, by (\ref{eq 34}) and (\ref{eq 35}),
$$
\psi_{p}^{1}( y_{p}) \ge \delta(\varepsilon) > 0, \quad
\psi_{p}^{1}(z_{p}) \ge \delta(\varepsilon) > 0, \quad 
$$
we have 
\begin{equation*}
 x_{p}^{1} \in (y_{p}, z_{p})  
\end{equation*}
and from (\ref{eq 36}), we get that 
$$0=c_{p} (\psi_{p}^{1})^{'}(x_{p}^{1}) \ge \alpha_{0}(\psi_{p}^{2}(x_{p}^{1})- \psi_{p}^{1}(x_{p}^{1})).$$
This gives that 
$$m_p^1=\psi_p^1(x^1_p)\ge \psi^2_p(x^1_p)\ge m^2_p\ge0.$$
Thus 
$$m^2_p\to 0\quad {\rm as}\quad p\to 0.$$
Using that 
$$
\psi_{p}^{2}( y_{p}) \ge \delta(\varepsilon) > 0 \quad{\rm and}\quad 
\psi_{p}^{2}(z_{p}) \ge \delta(\varepsilon) > 0, 
$$
we then get
\begin{equation}\label{eq 41}
x_{p}^{2} \in (y_{p}, z_{p}).  
\end{equation}
Then by \eqref{eq 36}, we have
\begin{equation}\label{eq 42}
\left\lbrace
\begin{array}{rcl}
 0=c_{p} (\psi_{p}^{1})^{'}(x_{p}^{1}) &\ge& \alpha_{0}(\psi_{p}^{2}(x_{p}^{1})- \psi_{p}^{1}(x_{p}^{1})) \\
 0=c_{p} (\psi_{p}^{2})^{'}(x_{p}^{2}) &\ge& 2(\, F(((\phi_{p}^{1})^{*}(x_{p}^{2}+a+r_{i}))_{i=0,...,N})- 
F(((\phi_{p}^{1})^{*}(x_{p}^{2}-a+r_{i}))_{i=0,...,N})) \\
&&+\alpha_{0}(\psi_{p}^{1}(x_{p}^{2})- \psi_{p}^{2}(x_{p}^{2})).
 \end{array}
\right.
\end{equation}
Using that 
$$\a_0(\psi_{p}^{1}(x_{p}^{2})- \psi_{p}^{2}(x_{p}^{2}))\ge \alpha_0 (\psi_{p}^{1}(x_{p}^{1})- \psi_{p}^{2}(x_{p}^{1}))\ge 0,$$
we deduce that
\begin{equation}\label{eq:10}
0\ge F(((\phi_{p}^{1})^{*}(x_{p}^{2}+a+r_{i}))_{i=0,...,N})- 
F(((\phi_{p}^{1})^{*}(x_{p}^{2}-a+r_{i}))_{i=0,...,N}).
\end{equation}

\paragraph{Step 3.3: $ \psi_{p}^{1}(x_{p}^{2}+r_{i}) \ge \psi_{p}^{1} (x_{p}^{1})=m_{p}^{1} $ for all $i$.}
Using (\ref{eq 41}) we get
\begin{equation}\label{eq 43}
b-\varepsilon \le (\phi^{1}_{p})^{*}( y_{p}) \le \phi^{1}_{p}(x_{p}^{2}) \le (\phi^{1}_{p})_{*}((z_{p}) \le  b+\varepsilon. 
\end{equation}
Using Lemma \ref{lem:1}, we deduce that there exists $\phi^1_0$ and $\phi^2_0$ such that
$$ \phi_{p}^{1}(x_{p}^{2}+\cdot) \to   \phi_{0}^{1}\quad{\rm and}\quad \phi_{p}^{2}(x_{p}^{2}+\cdot) \to   \phi_{0}^{2} \quad \mbox{a.e.\,on}\quad \mathbb{R} $$
and $(\phi^1_0,\phi^2_0)$ is solution of 
\begin{equation}\label{eq:02}
\left\lbrace
\begin{array}{lcl}
 c\,\,(\phi_0^{1})'(z) = \alpha_{0}(\phi_0^{2}(z)-\phi_0^{1}(z)) \\
c\,\,(\phi_0^{2})'(z)=2\,F((\phi_0^{1}(z+r_{i}))_{i=0,...,N})+\alpha_{0}(\phi_0^{1}(z)-\phi_0^{2}(z)).\\
\end{array}
\right.
\end{equation}
Using that 
\begin{equation}\label{eq 44}
0\le (\phi_{p}^{2})_{*}(x_{p}^{2}+a)-(\phi_{p}^{2})^{*}(x_{p}^{2}-a)=\psi_{p}^{2}(x_{p}^{2})\le \psi_{p}^{2}(x_{p}^{1})\le\psi_{p}^{1}(x_{p}^{1})=m_{p}^{1}\to 0\quad\mbox{as}\,\,p \to 0
\end{equation}
we deduce that $\phi_{0}^{2}=K_{1}$ on $(-a, a)$. Using Lemma \ref{lem:2}, we deduce that  $\phi_{0}^{1}=K_{1}$ on $(-a, a)$ with $K_1\in (b-\e,b+\e)$ (by  (\ref{eq 43})). 

The second equation of \eqref{eq:02} implies that $f(K_1)=0$, which gives $K_1=b$.
We then deduce (using the uniform Lipschitz continuity or Helly's Lemma in the case $c=0$), that 
$$ \sup_{(x_{p}^{2}-a+\delta, x_{p}^{2}+a-\delta)}  \left |   \phi_{p}^{1}(x)-b  \right | \to 0 \,\,\, \mbox{for\,\,all}\,\, \delta>0, $$
which implies
\begin{equation}\label{eq 46}
(\phi_{p}^{1})_{*}(x_{p}^{2}+a -\delta), (\phi_{p}^{1})^{*}(x_{p}^{2}-a+\delta) \to b\,\,\,\mbox{as}\,\,p \to 0
\end{equation}
Using \eqref{eq 32}, we deduce, taking $\delta$ small enough ($ 0< \delta \le a-r^{*}$), that
$$ y_{p}\le x_{p}^{2}+r_{i} \le z_{p} \,\,\mbox{ for all }\,\,i$$
which gives that
 \begin{equation}\label{eq 47}
\psi_{p}^{1}(x_{p}^{2}+r_{i}) \ge \psi_{p}^{1}(x_{p}^{1})=m_{p}^{1} . 
\end{equation}

\paragraph{Step 3.4: Getting a contradiction.} 
In this step, we assume that  $ m_{p}^{1} >0 $ and we want to get a contradiction.
Set
$$ k_{i}= \begin{cases}
(\phi_{p}^{1})_{*}(x_{p}^{2}+r_{i}+a)\,\,\,\mbox{if}\,\,r_{i} \le 0 \\
(\phi_{p}^{1})^{*}(x_{p}^{2}+r_{i}-a )\,\,\,\mbox{if}\,\,r_{i} > 0 
\end{cases}$$
and note that
$$ k_{i} \in [(\phi_{p}^{1})^{*}(x_{p}^{2}-a), (\phi_{p}^{1})_{*}(x_{p}^{2}+a)]$$
which implies that $k_i\to b$ as $p\to 0$ (by (\ref{eq 46})).

\noindent Hence from  (\ref{eq:10}),(\ref{eq 47})  and using the monotonicity of $ F $, we get
$$ 0 \ge F((a_{i})_{i=0,...,N})-F((c_{i})_{i=0,...,N})$$
where 
$$a_{i}= 
\left\{\begin{array}{lll}
k_{i}&\mbox{if}&r_{i} \le 0\\
 k_{i}+m_{p}^{1}&\mbox{if}& r_{i} > 0         
\end{array}
\right.$$
and 
$$c_{i}=
\left\{
\begin{array}{lll}
k_{i}-m_{p}^{1}&\mbox{if}&r_{i} \le 0\\
 k_{i}&\mbox{if}& r_{i} > 0.       
\end{array}
\right.
$$

\noindent Therefore from the fact that $k_i\to b$ and $ m_{p}^{1} \to 0 $, we deduce that
$$ a_{i} \to b   \quad \mbox{and} \quad  c_{i} \to b  \quad \mbox{as} \quad p \to 0 .$$
Since F is $C^{1}$ near $ \{b\}^{n+1}$ and $ c_{i}+t\,\,(a_{i}-c_{i})=c_{i}+t\,m_{p}^{1}$, we have
\begin{align*}
 0 \ge& \int_{0}^{1} dt\,\,\, \sum_{i=0}^{N}(a_{i}-c_{i})\left( \dfrac{\partial F}{\partial X_{i}}(c_{j}+t(a_{j}-c_{j})_{j=0,...,N})\right)\\
=&\,\, \int_{0}^{1} dt\,\,\,  \sum_{i=0}^{N}m_{p}^{1}\left(  \dfrac{\partial F}{\partial X_{i}}( (c_{j}+t\,\,m_{p}^{1})_{j=0,...,N})\right).
\end{align*}
Using that $m_{p}^{1} > 0$, we get 
\begin{align*}
0 \ge&  \int_{0}^{1}dt\,\,\, \sum_{i=0}^{N} \dfrac{\partial F}{\partial X_{i}}( (c_{j}+t\,\,m_{p}^{1})_{j=0,...,N})\\
=&\,f'(b)+ \int_{0}^{1} dt\,\,\,\left( \sum_{i=0}^{N}\dfrac{\partial F}{\partial X_{i}}((c_{j}+t\,\,m_{p}^{1})_{j=0,...,N})
-\sum_{i=0}^{N} \dfrac{\partial F}{\partial X_{i}}(b,...,b)\right) .
\end{align*}
But $F$ is $ C^{1} $ near $\{b\}^{N+1}$ and $c_{i}+t\,m_{p}^{1} \to b $ for all $ i$, thus
$$ \int_{0}^{1} dt\,\,\,\left( \sum_{i=0}^{N}\dfrac{\partial F}{\partial X_{i}}((c_{j}+t\,\,m_{p}^{1})_{j=0,...,N})-
\sum_{i=0}^{N} \dfrac{\partial F}{\partial X_{i}}(b,...,b)\right) \to 0 \quad \mbox{as} \quad p \to 0 .$$
This implies that 
$$ 0 \ge f'( b) >  0 $$ 
which is a contradiction with assumption (B). \medskip

\paragraph{Step 5: $ m_{p}^{1}>0$.}
We split this step into two cases:

\paragraph{Case 1: $ F $ is strongly increasing in some direction.}
We assume that F satisfies 
\begin{equation}\label{eq 48}
 \dfrac{\partial F }{\partial X_{i_{1}}} \ge \delta_{0} > 0 .
\end{equation} 
We assume by contradiction that $ m_{p}^{1}=0$. Thus
$$ \psi_{p}^{1}(x_{p}^{1})=(\phi_{p}^{1})_{*}(x_{p}^{1}+a)-(\phi_{p}^{1})^{*}(x_{p}^{1}-a)=0. $$
Since $ \phi_{p}^{1} $  is non-decreasing, we get
$$ \phi_{p\lvert _{ (x_{p}^{1}-a, x_{p}^{1}+a)}}^{1} =\phi_{p}^{1}(x_{p}^{1})=b .$$
The first equation of \eqref{eq:01} implies that
$$ \phi_{p\lvert _{ (x_{p}^{1}-a, x_{p}^{1}+a)}}^{2} =b .$$
Let  $ d_{1}\ge x_{p}^{1}+a $  be the first real number such that 
$$ \phi_{p}^{1}(d_{1}+\eta_{1})>b\,\,\mbox{for\,\,every}\,\,\eta_{1}>0.$$
We choose $ 0<\eta_{1}<r_{i_{1}}$ and set
$$ x_{1}= d_{1}+\eta_{1}-r_{i_{1}}. $$ 
From the definition of  $ d_{1}$, we deduce that 
$$\phi_{p}^{1}= \phi_{p}^{2}=b \textrm{ on a neighborhood of }x_{1}.$$ 
hence $ (\phi^{2}_{p})^{'}(x_{1})=0 $. Moreover, we have 
$$\begin{cases}
 \phi_{p}^{1}(x_{1}+r_{i}) \ge b\,\,\,\mbox{for\,\,all}\,\,i\neq i_{1} \\
\phi_{p}^{1}(x_{1}+r_{i_{1}})=\phi_{p}^{1}(d_{1}+\eta_{1})>b \,\,\mbox{for}\,\,i=i_{1} .\\
\end{cases}$$
Thus, the second equation of \eqref{eq:01} implies that
\begin{align*}
0=c\,(\phi_{p}^{2})^{'}(x_{1})=&2F((\phi_{p}^{1}(x_{1}+r_{i}))_{i=0,....,N})\\
\ge&2 F(b,...,\phi_{p}^{1}(x_{1}+r_{i_{1}}),....,b)\\
\ge &f(b)+\delta_{0}(\phi_{p}^{1}(d_{1}+\eta_{1})-b)\\
=&\delta_{0}(\phi_{p}^{1}(d_{1}+\eta_{1})-b)>0 .
\end{align*}
This is a contradiction. \medskip

\paragraph{Case 2: Create the monotonicity.}
In fact, we can always assume (\ref{eq 48}) for a modification $ F_{p} $ of  $F$, where
$$F_{p}(X_{0},X_{1},...,X_{N})=F(X_{0},X_{1},...,X_{N})+p(X_{i_{1}}-X_{0}).$$
Then the whole construction works for $F$ replaced by $F_{p}$
with the additional monotonicity property (\ref{eq 48}) with $\delta_{0}=p.$
Once we pass to the limit $p\to 0,$ we still get the
same contradiction as in Step $3.4
$ and we recuperate the construction
of traveling wave $(\phi_1,\phi_2)$ of (\ref{eq 4}) for the function $F.$

\end{proof}
%%%%%%%%%%%%%%%%%%%%%%%%%%%%%%%%%%%%%%%%%%%%%%%%
%%%%%%%%%%%%%%%%%%%%%%%%%%%%%%%%%%%%%%%%%%%%%%%%
\section{ Uniqueness of the velocity $ c $ }\label{sec:4}
%%%%%%%%%%%%%%%%%%%%%%%%%%%%%%%%%%%%%%%%%%%%%%%%
In the first subsection, we prove a comparison principle on $(-\infty, r^{*}])$ and then  another one on  $[-r^{*}, +\infty))$. 
These two comparison principles will be used in the second subsection to prove the uniqueness of the velocity.
%%%%%%%%%%%%%%%%%%%%%%%%%%%%%%%%%%%%%%%%%%%%%%%%
\subsection{ Comparison principle on the half-line }
\begin{theo}[Comparison principle on $(-\infty, r^{*}{]}$]\label{th 2}
Let $ F: [0,1]^{N+1} \to \mathbb{R} $ satisfying $ (A) $ and assume that 
\begin{equation}\label{eq 49}
\left\lbrace
\begin{array}{lcl}
 \textrm{there exists }\beta_{0}>0 \textrm{ such that if}:\\
   Y=(Y_{0},....,Y_{N}), Y+(a....,a) \in [0, \beta_{0}^{N+1}] \\
  \mbox{then}\,\,
F(Y+(a,...,a)) < F(Y)\,\,\mbox{if}\,\,a >0.
\end{array}
\right.
\end{equation}
 Let $ (u_{1}, u_{2}) $ and $ (v_{1}, v_{2}): (-\infty, r^{*}]^{2} \to [0,1]^{2} $  be respectively a sub and a super-solution of
\begin{equation}\label{eq 50}
\left\lbrace
\begin{array}{lcl}
c\,\,\,u_{1}'(z) = \alpha_{0}(u_{2}(z)-u_{1}(z))&{on}&(-\infty,0) \\
c \,\,\,u_{2}'(z)=2\,F((u_{1}(z+r_{i}))_{i=0,...,N})+\alpha_{0}(u_{1}(z)-u_{2}(z))&\mbox{on} &(-\infty,0)
 \end{array}
\right.
\end{equation}
We suppose that 
 $$ \begin{cases}
 u_{1} \le \beta_{0}\,\,\mbox{on}\,\,(-\infty, 2r^{*}]\\
 u_{2} \le \beta_{0}\,\,\mbox{on}\,\,(-\infty, 2r^{*}]\\
\end{cases}\quad
{\rm and}\quad 
 \begin{cases}
   u_{1} \le v_{1}\,\,\mbox{on}\,\,[0, r^{*}] \\
   u_{2} \le v_{2}\,\,\mbox{on}\,\,[0, r^{*}].\\
\end{cases}$$
Then 
$$ \begin{cases}
   u_{1} \le v_{1}\,\,\mbox{on}\,\,(-\infty, r^{*}] \\
   u_{2} \le v_{2}\,\,\mbox{on}\,\,(-\infty, r^{*}]. \\
\end{cases}$$
\end{theo}

\begin{cor}[Comparison principle on ${[}-r^{*}, +\infty)$]\label{cor 2} 
Let $ F: [0,1]^{N+1} \to \mathbb{R} $ satisfying $ (A) $ and assume that
\begin{equation}\label{eq 51}
\left\lbrace
\begin{array}{lcl}
  \textrm{there exists }\beta_{0} > 0 \textrm{ such that if}:\\
   Y=(Y_{0},....,Y_{N}), Y+(a,....,a) \in [1-\beta_{0},1]^{N+1} \\
  \mbox{then}\,\,
F(Y+(a,...,a)) < F(Y)\,\,\mbox{if}\,\,a >0.
\end{array}
\right.
\end{equation}
Let $ (u_{1}, u_{2}) $ and $ (v_{1}, v_{2}): [-r^{*}, +\infty)^{2} \to [0,1]^{2} $ be respectively a sub 
and a super-solution of (\ref{eq 50})   on $(0,+\infty)$.
We assume that 
$$ \begin{cases}
 v_{1} \ge 1-\beta_{0}\,\,\mbox{on}\,\,[-2r^{*},+\infty) \\
 v_{2} \ge 1-\beta_{0}\,\,\mbox{on}\,\,[-2r^{*},+\infty)
 \end{cases} 
 \quad
{\rm and}\quad 
\begin{cases}
   u_{1} \le v_{1}\,\,\mbox{on}\,\,[-r^{*},0] \\
   u_{2} \le v_{2}\,\,\mbox{on}\,\,[-r^{*},0].
\end{cases}$$
Then 
$$ \begin{cases}
   u_{1} \le v_{1}\,\,\mbox{on}\,\,[-r^{*}, +\infty] \\
   u_{2} \le v_{2}\,\,\mbox{on}\,\,[-r^{*}, +\infty] .\\
\end{cases}$$
 \end{cor}
\begin{lem}[Transformation of a solution of (\ref{eq 50})]
\label{lemma 8} 
Let $ (u_{1}, u_{2}), (v_{1}, v_{2}): (-\infty, r^{*}]^{2} \to [0,1]^{2}$ be respectively a sub and a super-solution of (\ref{eq 50}).Then 
$$ \begin{cases}
  \hat{u}_{1}(x)=1-u_{1}(-x) \\
  \hat{u}_{2}(x)=1-u_{2}(-x) 
   \end{cases}
\,\,\mbox{and}\,\,
\begin{cases}
    \hat{v}_{1}(x)=1-v_{1}(-x) \\
    \hat{v}_{2}(x)=1-v_{2}(-x)
\end{cases}
$$
are respectively a super and a sub-solution of (\ref{eq 50})  on $ (0 +\infty) $ with $F,\; c,\; r_{i} $ (for all $ i \in \{0,...,N \} $)
 replaced by  $ \hat{F},\; \hat{c}$ and $\hat{r} $, given by
\begin{equation}\label{eq 52}
\left\lbrace
\begin{array}{lcl}
   \hat{F}(X_{0}, .... , X_{N})=-F(1-X_{0}, ...., 1-X_{N}) \\
   \hat{c}=-c \\
   \hat{r}_{i}=-r_{i}
\end{array}
\right.
\end{equation}
with $ \hat{F}:[0, 1]^{N+1} \to \mathbb{R} $ satisfying $(A)$, $(B)$ and $(C)$, where $ b $ and $ f $ are replaced by
$$ \begin{cases}
    \hat{b}=1-b \\
\hat{f}(v)=-f(1-v). 
   \end{cases}$$
\end{lem}
%----------------------------
\begin{proof}
Let $ ( u_{1}, u_{2}): (-\infty, r^{*}]^{2} \to [0, 1]^{2} $ be a sub-solution of  (\ref{eq 50}) and set 
$$\hat {u}_{1}(x)=1-u_{1}(-x),\quad\hat {u}_{2}(x)=1-
u_{2}(-x).$$
We have 
$$ \begin{cases}
 c\,\hat {u}_{1}'(x)=c\,u_{1}'(-x) \le \alpha_{0} (u_{2}(-x)-u_{1}(-x))\\
c\,\hat {u}_{2}'(x)=c\,u_{2}'(-x) \le -2\,F((u_{1}(-x+r_{i}))_{i=0,...,N})+\alpha_{0}(u_{1}(-x)-u_{2}(-x))
\end{cases}$$
thus 
$$ \begin{cases}
\hat{c}\,\hat {u}_{1}'(x)\ge \alpha_{0} (\hat {u}_{2}(x)-\hat {u}_{1}(x))\\
 \hat c\,\hat {u}_{2}'(x)\ge 2\,F((1-\hat{u}_{1}(x-r_{i}))_{i=0,...,N})+\alpha_{0}(\hat{u}_{2}(x)-\hat{u}_{1}(x)).
\end{cases}$$
Hence $(\hat{u}_{1}, \hat{u}_{2})$ is a super-solution of   (\ref{eq 50}) on $(0, +\infty)$. Similarly, we prove that
$(\hat{v}_{1}, \hat{v}_{2})$ is a sub-solution of the same equation on $(0, +\infty)$.

\end{proof}

\noindent The proof of Corollary \ref{cor 2} is now very easy.
%----------------------------------------
\begin{proof}[Proof of Corollary \ref{cor 2}]
Let $ (u_{1}, u_{2}),\,(v_{1}, v_{2}): [-r^{*}, +\infty)^{2} \to [0,1]^{2} $ be a sub and a super-solution of (\ref{eq 50}) 
on $ (0, +\infty) $ such that $ v_{1} \ge 1-\beta_{0}$ and $ v_{2} \ge 1-\beta_{0}$ on $ [-2r^{*}, +\infty)$. We set
$$ \hat{u}_{1}(x)=1-u_{1}(-x), \quad \hat{u}_{2}(x)=1-u_{2}(-x),\quad \hat{v}_{1}(x)=1-v_{1}(-x) \quad \mbox{and} \quad \hat{v}_{2}(x)=1-v_{2}(-x). $$
Thus $ \hat{v}_{1},\hat{v}_{2} \le \beta_{0} $ on $ (-\infty, 2r^{*}]$ and by Lemma 4.3   $ (\hat{u}_{1}, \hat{u}_{2})$ and $(\hat{v}_{1}, \hat{v}_{2})$ 
 are respectively a super and a sub-solution of  (\ref{eq 50}).  
Since $ F $ satisfies (\ref{eq 51}), we deduce that $ \hat{F} $ satisfies (\ref{eq 49}).
By Theorem \ref{th 2}, we then get that
$$\begin{cases}
   \hat{v}_{1} \le \hat{u}_{1} \quad \mbox{on} \quad (-\infty, r^{*}] \\
   \hat{v}_{2} \le \hat{u}_{2} \quad \mbox{on} \quad (-\infty, r^{*}].
  \end{cases}$$
This implies that 
 $$\begin{cases}
   u_{1} \le v_{1} \quad \mbox{on} \quad [-r^{*},+\infty) \\
   u_{2} \le v_{2}  \quad \mbox{on} \quad [-r^{*},+\infty).
  \end{cases}$$
  
\end{proof}
%---------------------------------
\noindent We now go back to the proof of Theorem \ref{th 2}
\begin{proof}[Proof of Theorem \ref{th 2}]
Let $ (u_{1}, u_{2})$ and $ (v_{1},v_{2}): (-\infty, r^{*}]^{2} \to [0,1]^{2}$ be respectively a sub and a super-solution of (\ref{eq 50}) such that
$$ \begin{cases}
    u_{1} \le \beta_{0}\quad \mbox{on}\quad (-\infty, 2r^{*}] \\
    u_{2} \le \beta_{0}\quad \mbox{on}\quad (-\infty, 2r^{*}]
   \end{cases}\quad
{\rm and}\quad 
 \begin{cases}
    u_{1} \le v_{1}\quad \mbox{on}\quad [0,r^{*}] \\
    u_{2} \le v_{2}\quad \mbox{on}\quad   [0,r^{*}]
\end{cases}$$ 

\paragraph{Step 0: Introduction.}
Let 
$$ \begin{cases}
    \bar v_{1}= \min(v_{1}, \beta_{0})\\
    \bar v_{2}=\min(v_{2}, \beta_{0})
   \end{cases}$$
According to (\ref{eq 49}) we have 
$$ F(\beta_{0},...,\beta_{0}) \le 0 . $$
Therefore the constant $ \beta_{0} $ is a super-solution of (\ref{eq 50}) and then
 $ (\bar v_{1}, \bar v_{2}) $ is a super-solution of (\ref{eq 50}) on $ (-\infty,0) $
with  $ u_{1} \le \bar{v}_{1}, u_{2} \le \bar{v}_{2}\,\,\mbox{ on}\,\,[0,r^{*}]$.
Moreover, since  $ \bar{v}_{1} \le v_{1}$ and $ \bar{v}_{2} \le v_{2} $, it is sufficient to prove the comparison
principle  (Theorem \ref{th 2})  between  $ (u_{1}, u_{2})$ and
$( \bar{v}_{1}, \bar{v}_{2})$ with $ u_{1}, \bar{v}_{1}, u_{2}, \bar{v}_{2} \in [0, \beta_{0}] $.
For simplicity, we note $(\bar{v}_{1}, \bar{v}_{2})$ as  $(v_{1}, v_{2})$.

\paragraph{Step 1: Doubling the variables.} 
We assume by contradiction that
$$  M=\sup_{x\in (-\infty,r^*]} \max(u_{1}(x)-v_{1}(x),u_{2}(x)-v_{2}(x)) > 0 $$ 
Let $0< \varepsilon, \alpha <1$ and we define 
\begin{align*}
& \psi_1(x,y)= u_{1}(x)-v_{1}(y)-\frac{ \left| x-y \right|^{2}}{2\varepsilon}+\alpha  x  \\
 &\psi_2(x,y)= u_{2}(x)-v_{2}(y)-\frac{ \left| x-y \right|^{2}}{2\varepsilon}+\alpha  x 
 \end{align*}
and
$$ M_{\varepsilon, \alpha}=  \sup_{x, y \in (-\infty,r^{*}] } \max(\psi_1(x,y),\psi_2(x,y))$$
Since the function $ \psi_1$ and $\psi_2$ are upper semi-continuous and satisfy
$ \psi_1(x,y),\psi_2(x,y) \to -\infty $ as $ \left| (x,y) \right| \to +\infty $, we deduce that $\psi_1$ and $\psi_2$ reach their maximum respectively at $(x_\e^1,y_\e^1)$ and $(x_\e^2,y_\e^2)\in (-\infty,r^*]^2$. We also denote by $(x_\e,y_\e,i_\e)\in (-\infty,r^*]^2\times \{1,2\}$ such that
$$M_{\e,\a}=\psi_{i_\e}(x_\e,y_\e).$$
Moreover, for $\a$ small enough, we get that
$$ M_{\varepsilon, \alpha} \ge \frac{  M }{2} >0 .$$
Using also the fact that    $ u_{1}(x_{\varepsilon}^1)-v_{1}({x^1_{\varepsilon}}) \le \beta_{0}$ and $u_{2}({x^2_{\varepsilon}})-v_{2}({x^2_{\varepsilon}}) \le \beta_{0}$, we get then  
\begin{equation}\label{eq 53}
 \dfrac{ \left| {x}_{\varepsilon}- {y}_{\varepsilon} \right|^{2}}{ 2\varepsilon}- \alpha  {x}_{\varepsilon}  \le \beta_{0}.
\end{equation}

\paragraph{Step 2: for all $\alpha$ and $ \varepsilon $ small enough, we have $ {x}_{\varepsilon},  {y}_{\varepsilon} \in (-\infty,0)$. } 
By contradiction, assume that there exists $ \alpha $ small enough and $ \varepsilon \to 0 $ such that 
 $ {x}_{\varepsilon} \in [0, r^{*}]$ or $ {y}_{\varepsilon} \in [0, r^{*}]$ and $i_\e=i$. We suppose that  $ {x}_{\varepsilon} \in [0, r^{*}]$ (the case $ {y}_{\varepsilon} \in [0, r^{*}]$  being similar). Using  (\ref{eq 53}), we deduce that
 $ \bar{y}_{\varepsilon} \in [- \sqrt{2 (\beta_{0}+r^*) \varepsilon}, r^{*}]$. Then $ {x}_{\varepsilon} $ and $ {y}_{\varepsilon} $ converge to
$ {x}_{0} \in [0,r^{*}] $ as  $ \varepsilon \to 0 $.
We deduce that 
$$ 0 < \dfrac{M}{2} \le \limsup_{ \varepsilon \to 0}(u_{i}({x}_{\varepsilon})-v_{i}({y}_{\varepsilon})) \le u_{i}({x}_{0})-v_{i}({y}_{0}) \le 0 $$
which is a contradiction.

\paragraph{Step 3: Viscosity inequalities.}
Using that $x\mapsto \psi_1(x,y_\e^1)$ reaches a maximum at point $x_\e^1$, we get that
$$c\,\left( \frac{ {x}^1_{\varepsilon}-{y}^1_{\varepsilon}}{\varepsilon}+\alpha\right)\le \alpha_{0} (u_{2}({x}^1_{\varepsilon})-u_{1}({x}^1_{\varepsilon})).$$
In the same way, we have
$$c\,\left( \frac{ {x}^1_{\varepsilon}-{y}^1_{\varepsilon}}{\varepsilon}\right)\ge \alpha_{0} (v_{2}({y}^1_{\varepsilon})-v_{1}({y}^1_{\varepsilon})).$$
Subtracting the two inequalities, we deduce that
\begin{equation}\label{eq:20}
c\alpha\le \a_0\left((u_{2}({x}^1_{\varepsilon})-v_{2}({y}^1_{\varepsilon}))-(u_{1}({x}^1_{\varepsilon})-v_{1}({y}^1_{\varepsilon}))\right)
\end{equation}
In the same way (using $\psi_2$) we get that
\begin{align}\label{eq 56}
c\alpha \le& 2\Big[(F((u_{1}({x}^2_{\varepsilon}+r_{i}))_{i=0,...,N})-F((v_{1}({y}^2_{\varepsilon}+r_{i}))_{i=0,...,N})\Big]\\
&+\alpha_{0} ((u_{1}({x}^2_{\varepsilon})-v_{1}({y}^2_{\varepsilon}))-(u_{2}({x}^2_{\varepsilon})-v_{2}({y}^2_{\varepsilon}))\nonumber
\end{align}

\paragraph{Step 4: Passing to the limit $ \e, \alpha \to 0 $.}
We set 
$$u_{j,i}^{\e,k}=u_j(x_\e^k+r_i),\quad {\rm and }\quad v_{j,i}^{\e,k}=v_j(x_\e^k+r_i).$$

The proof is split into two cases:

\paragraph{Case 1: $\exists\; \e,\a\to 0$ such that $i_\e=1$.}
In that case, equation \eqref{eq:20} implies that
$$\psi_2(x_\e^2,y_\e^2)\ge \psi_2(x_\e^1,y_\e^1)\ge\psi_1(x_\e^1,y_\e^1) +\frac {c\a}{\a_0}\ge \frac M2+\frac {c\a}{\a_0}\ge \frac M4$$
for $\a$ small enough. Hence, using classical arguments, we deduce that
\begin{equation}\label{eq:21}
  \dfrac{ \left| {x}^2_{\varepsilon}- {y}^2_{\varepsilon} \right|^{2}}{ 2\varepsilon},  \alpha  {x}^2_{\varepsilon}\to 0 \quad{\rm as} \quad \e,\a\to0.
\end{equation}
Moreover, using that $\psi_1(x^2_\e+r_i, y_\e^2+r_i)\le \psi_1(x^1_\e, y_\e^1)$, we have
$$ u_{1,i}^{\e,2}\le v_{1,i}^{\e,2}+m_\e+\delta_i^\e$$
where $m_\e=u_1(x_\e^1)-v_1(y^1_\e)$, and $\delta_i^\e=\frac{|x^2_\e-y_\e^2|^2}{2\e}-\a (x^2_\e+r_i)$ (note that $\a x^1_\e\le 0$).
Since $ u_{j,i}^{\e,k}, v_{j,i}^{\e,k} \in [0, \beta_{0}]$ and $ \frac{M^{1}}{2} \le m_{\e} \le \beta_{0}$, we deduce that as 
$ \e, \alpha \to 0 $
$$ \begin{cases}
 u_{j,i}^{\e,k}\to u_{1,i}^k\\
v_{j,i}^{\e,k}\to v_{1,i}^k\\
    m_{\e} \rightarrow m_{0}=u_{1,0}^1-v_{1,0}^1\\
    \delta_{i}^{\e} \to 0
   \end{cases}$$
with  $ u_{j,i}^k, v_{j,i}^k \in [0,\beta_{0}]$, 
$0< \frac{M^{1}}{2} \le m_{0} \le \beta_{0} $ and 
$$u_{1,i}^2\le v_{1,i}^2+m_0.$$
Passing  to the limit in (\ref{eq 56}) implies that 
$$
0\le 2\,(F((u_{1,i}^{2})_{i=0,...,N})-F((v_{1,i}^{2})_{i=0,...,N}))+\alpha_{0}((u_{1,0}^2-v_{1,0}^2)-(u_{2,0}^2-v_{2,0}^2)). 
$$
We define $\bar m=m_0-(u_{1,0}^2-v_{1,0}^2)\ge 0$.
Thus, using the monotony \eqref{eq:monV0}, we get
$$0\le 2\,(F(u_{1,0}^2+\bar m, (u_{1,i}^{2})_{i=1,...,N})-F((v_{1,i}^{2})_{i=0,...,N}))+\alpha_{0}( (u_{1,0}^1-v_{1,0}^1)-(u_{2,0}^2-v_{2,0}^2)). 
$$
Passing to the limit in \eqref{eq:20} and using the fact that $u^1_{2,0}-v^1_{2,0}\le u^2_{2,0}-v^2_{2,0}$, we get that
$$\a_0( (u_{1,0}^1-v_{1,0}^1)-(u_{2,0}^2-v_{2,0}^2))\le 0.$$
Thus, up to redefine $u_{1,0}^2$ by $u_{1,0}^2+\bar m \in [0,\beta_0]$, we get
\begin{equation}\label{eq 57}
0\le F(u_{1,0}^2, (u_{1,i}^{2})_{i=1,...,N})-F((v_{1,i}^{2})_{i=0,...,N}). 
\end{equation}
with $ u_{j,i}^k, v_{j,i}^k \in [0,\beta_{0}]$, 
$0< \frac{M^{1}}{2} \le m_{0} \le \beta_{0} $,
$$u_{1,i}^2\le v_{1,i}^2+m_0\quad {\rm and}\quad u^2_{1,0}=v^2_{1,0}+m_0.$$

\paragraph{Case 2: $\exists\; \e,\a\to 0$ such that $i_\e=2$.} Using that $\psi_1(x_\e^2+r_i,y_\e^2+r_i)\le \psi_2(x_\e^2,y_\e^2)$, we get that
$$u_{1,i}^{\e,2}\le v_{1,i}^{\e,2}+m_\e+\delta_i^\e$$
where $m_\e=u_2(x_\e^2)-v_2(y^2_\e)$, and $\delta_i^\e=-\a r_i$.
Since $ u_{j,i}^{\e,k}, v_{j,i}^{\e,k} \in [0, \beta_{0}]$ and $ \frac{M^{1}}{2} \le m_{\e} \le \beta_{0}$, we deduce that as 
$ \e, \alpha \to 0 $
$$ \begin{cases}
 u_{j,i}^{\e,k}\to u_{1,i}^k\\
v_{j,i}^{\e,k}\to v_{1,i}^k\\
    m_{\e} \rightarrow m_{0}=u_{2,0}^2-v_{2,0}^2\\
    \delta_{i}^{\e} \to 0
   \end{cases}$$
with  $ u_{j,i}^k, v_{j,i}^k \in [0,\beta_{0}]$, 
$0< \frac{M^{1}}{2} \le m_{0} \le \beta_{0} $ and 
$$u_{1,i}^2\le v_{1,i}^2+m_0.$$
Passing  to the limit in (\ref{eq 56}) implies that 
$$
0\le 2\,(F((u_{1,i}^{2})_{i=0,...,N})-F((v_{1,i}^{2})_{i=0,...,N}))+\alpha_{0}((u_{1,0}^2-v_{1,0}^2)-(u_{2,0}^2-v_{2,0}^2)). 
$$
We define $\bar m=m_0-(u_{1,0}^2-v_{1,0}^2)\ge 0$.
Thus, using the monotony \eqref{eq:monV0}, we get
\begin{align*}
0\le& 2\,(F(u_{1,0}^2+\bar m, (u_{1,i}^{2})_{i=1,...,N})-F((v_{1,i}^{2})_{i=0,...,N}))+\alpha_{0}( (u_{1,0}^2-v_{1,0}^2)-(u_{2,0}^2-v_{2,0}^2)+\bar m)\\
=&2\,(F(u_{1,0}^2+\bar m, (u_{1,i}^{2})_{i=1,...,N})-F((v_{1,i}^{2})_{i=0,...,N}))
\end{align*}
Thus, up to redefine $u_{1,0}^2$ by $u_{1,0}^2+\bar m \in [0,\beta_0]$, we get
\begin{equation}\label{eq 57bis}
0\le F((u_{1,i}^{2})_{i=0,...,N})-F((v_{1,i}^{2})_{i=0,...,N}). 
\end{equation}
with $ u_{j,i}^k, v_{j,i}^k \in [0,\beta_{0}]$, 
$0< \frac{M^{1}}{2} \le m_{0} \le \beta_{0} $,  
$$u_{1,i}^2\le v_{1,i}^2+m_0\quad {\rm and}\quad u^2_{1,0}=v^2_{1,0}+m_0.$$

\paragraph{Step 5: Getting a contradiction.}

We claim that for all $i,$ there exists
$l_{i},\,l'_{i}\geq 0$ such that
\begin{equation}\label{eq29}
u_{1,i}^2+l_{i}=v_{1,i}^2-l'_{i}+m_{0},
\end{equation}
and
$$\left\{
\begin{aligned}
&\overline{u}_{1,i}^2:=u_{1,i}^2+l_{i}\leq \beta_{0}\\
&\overline{v}_{1,i}^2:=v_{1,i}^2-l'_{i}\geq 0.
\end{aligned}
\right.$$
Recall that for all $i\in\{0,...,N\},$ we have
$$\left\{
\begin{aligned}
&u_{1,i}^2,\,v_{1,i}^2\in[0,\beta_{0}]\\
&u_{1,i}^2\leq v_{1,i}^2+m_{0}\\
&u_{1,0}^2-v_{1,0}^2=m_{0}\leq\beta_{0}.
\end{aligned}
\right.$$
If for some $i,$ $u_{1,i}^2=v_{1,i}^2+m_{0},$ then it suffices to take $l_i=l_i'=0$.
Assume then that $u_{1,i}^2<v_{1,i}^2+m_{0}$.

\paragraph{Case 1: $u_{1,i}^2,\,v_{1,i}^2\in(v_{1,0}^2,u_{1,0}^2)$.}
Set $l_{i}=u_{1,0}^2-u_{1,i}^2$ and $l'_{i}=v_{1,i}^2-v_{1,0}^2$. Then
$$\left\{
\begin{aligned}
&\overline{u}_{1,i}^2=u_{1,i}^2+l_{i}=u_{1,0}^2\leq\beta_{0}\\
&\overline{v}_{1,i}^2=v_{1,i}^2-l'_{i}=v_{1,0}^2\geq 0,
\end{aligned}
\right.$$
and $\overline{u}_{1,i}^2=\overline{v}_{1,i}^2+m_{0}.$

\paragraph{Case 2: $u_{1,i}^2>u_{1,0}^2$ and $v_{1,i}^2>v_{1,0}^2$.}
Since $u_{1,i}^2-v_{1,0}^2>m_{0},$ then there exists $l'_{i}<v_{1,i}^2-v_{1,0}^2$ such that
$$u_{1,i}^2=v_{1,i}^2-l'_{i}+m_{0}$$ and $\overline{v}_{1,i}^2=v_{1,i}^2-l'_{i}>v_{1,0}^2\geq 0.$
Thus, it is sufficient to take $l_{i}=0.$

\paragraph{Case 3: $u_{1,i}^2<u_{1,0}^2$ and $v_{1,i}^2<v_{1,0}^2$.}
This case can be treated as Case 2 by taking $l'_{i}=0$ and $l_{i}<u_{1,0}^2-u_{1,i}^2.$
\bigskip

Finally, going back to (\ref{eq 57}) or \eqref{eq 57bis}, since $F$ is non-decreasing, we deduce that
\begin{eqnarray*}
0&\leq&F((u_{1,i}^2)_{i=0,...,N})-F((v_{1,i}^2)_{i=0,...,N})\\
&\leq&F((\overline{u}_{1,i}^2)_{i=0,...,N})-F((\overline{v}_{1,i}^2)_{i=0,...,N})\\
&=&F((\overline{u}_{1,i}^2)_{i=0,...,N})-F((\overline{u}_{1,i}^2-m_{0})_{i=0,...,N})\\
&<&0.
\end{eqnarray*}
Last inequality takes place since $F$ verifies (\ref{eq 49}) for
$\overline{u}_{1,i}^2,\,\overline{u}_{1,i}^2-m_{0}\in[0,\beta_{0}]$ and $m_{0}>0.$ Therefore, we get a contradiction.

\end{proof}

%%%%%%%%%%%%%%%%%%%%%%%%%%%%%%%%%%%%%%%%%%%%%%%%%%%%%
\subsection{ Uniqueness of the velocity}
In this subsection, we use Theorem \ref{th 2} and Corollary \ref{cor 2} in order to prove the uniqueness of the velocity $c$.

\begin{pro}[Uniqueness of the velocity]\label{pro 4}
Under assumptions $(A)$, we consider the function $ F $ defined on $[0,1]^{N+1} $. 
Let $(c_{1}, ( \phi_{11}, \phi_{12}))$ and  $(c_{2}, ( \phi_{21}, \phi_{22}))$  be two solutions of (\ref{eq 5}), with $\phi_{11},\phi_{12},\phi_{21},\phi_{22}:\R\to [0,1]$.
If $ F $ satisfies in addition $ (C) $, then $ c_{1} = c_{2} $.
\end{pro}

\begin{proof}
Assume that  $(c_{1}, (\phi_{11}, \phi_{12}))$ and  $(c_{2}, (\phi_{21}, \phi_{22}))$ are solutions of (\ref{eq 5}) and assume by contradiction that
$ c_{1} < c_{2} $. 
We have 
$$\begin{cases}
\phi_{11}(-\infty)=0, \,\,\phi_{11}(+\infty)=1 \\
\phi_{12}(-\infty)=0, \,\,\phi_{12}(+\infty)=1. 
\end{cases} 
\quad \mbox{and} \quad
\begin{cases}
\phi_{21}(-\infty)=0, \,\,\phi_{21}(+\infty)=1 \\
\phi_{22}(-\infty)=0, \,\,\phi_{22}(+\infty)=1. 
\end{cases} $$
We set $ \delta= \min(\beta_{0}, \frac{1}{4}) $ where $ \beta_{0}$ is given in assumption $ (C) $ and up to translate $(\phi_{11},\phi_{12})$ and $(\phi_{21},\phi_{22})$, we assume that 
$$ \begin{cases}
    \phi_{11}(x) \ge 1- \delta \quad \forall  \quad x \ge -2r^{*} \\
   \phi_{12}(x) \ge 1- \delta  \quad \forall   \quad x \ge -2r^{*} 
   \end{cases} $$
and 
$$ \begin{cases}
 \phi_{21}(x) \le  \delta \quad \forall  \quad x \le 2r^{*} \\
\phi_{22}(x) \le  \delta \quad \forall    \quad x \le2 r^{*}.
   \end{cases} $$
This implies that 
$$ \begin{cases}
 \phi_{21}(x) \le \phi_{11}(x) \quad \mbox{over} \quad [-r^{*}, r^{*}] \\
 \phi_{22}(x) \le \phi_{12}(x)  \quad  \mbox{over} \quad [-r^{*}, r^{*}].
 \end{cases} $$
Moreover, since $ c_{1} < c_{2} $, we have
$$ \begin{cases}
  c_{1} \,\, (\phi_{21})'(x) \le c_{2} \,\, (\phi_{21})'(x)= \alpha_{0} (\phi_{22}(x)-\phi_{21}(x)) \\
  c_{1} \,\,  (\phi_{22})'(x) \le c_{2} \,\,  (\phi_{22})'(x)= 2\,F((\phi_{21}(x-r_{i})_{i=0,...,N})+\alpha_{0}( \phi_{21}(x)- \phi_{22}(x)).
   \end{cases}$$
Thus $ (c_{1},(\phi_{21},\phi_{22})) $  is a sub-solution of (\ref{eq 5}). Using Corollary \ref{cor 2}, we deduce that 
$$\begin{cases}
   \phi_{21} \le \phi_{11}\,\,\mbox{over} \,\,[-r^{*}, +\infty) \\
   \phi_{22} \le \phi_{12}\,\,\mbox{over}\,\,[-r^{*}, +\infty).
  \end{cases} $$
Similarly, using Theorem \ref{th 2}, we get that
$$\begin{cases}
   \phi_{21} \le \phi_{11} \,\,\mbox{over} \,\, (-\infty, r^{*}] \\
   \phi_{22} \le \phi_{12} \,\,\mbox{over}\,\,  (-\infty, r^{*}] .
  \end{cases} $$
Therefore 
$$\begin{cases}
   \phi_{21} \le \phi_{11} \,\,\mbox{over} \,\, \mathbb{R} \\
   \phi_{22} \le \phi_{12} \,\, \mbox{over} \,\,  \mathbb{R}.
  \end{cases} $$
We set 
$$\begin{cases}
 u_{1}(t,x)=\phi_{11}(x+c_{1}t)\\
 u_{2}(t,x)=\phi_{12}(x+c_{1}t)\\
 u_{3}(t,x)=\phi_{21}(x+c_{2}t)\\
 u_{4}(t,x)=\phi_{22}(x+c_{2}t).
  \end{cases} $$
Then for $i=1, j=2 $ and $ i=3, j=4 $, we have
\begin{equation}\label{eq 60}
\left\lbrace
\begin{array}{lcl}
  \partial_{t} u_{i}(t,x)= \alpha_{0} (u_{j}(t,x)-u_{i}(t,x)) \\
  \partial_{t} u_{j}(t,x)= 2\,F((u_{i}(t,x+ri))_{i=0,...,N})+\alpha_{0}(u_{i}(t,x)-u_{j}(t,x)).
\end{array}
\right.
\end{equation}
Moreover, at time $ t=0 $, 
\begin{equation}\label{eq 61}
\left\lbrace
\begin{array}{lcl}
  u_{1}(0,x)=\phi_{11}(x) \ge \phi_{21}(x)=u_{3}(0,x)\,\, \mbox{over} \,\, \mathbb{R} \\
  u_{2}(0,x)=\phi_{12}(x) \ge \phi_{22}(x)=u_{4}(0,x)\,\,  \mbox{over} \,\, \mathbb{R}.
\end{array}
\right.
\end{equation}
Then, applying the comparison principle for equation (\ref{eq 60}), we get 
 $$\begin{cases}
 u_{1} \ge u_{3}\,\,\forall\,\,t\ge 0 \,\, \forall \, x \in \mathbb{R} \\
 u_{2} \ge u_{4}\,\,\forall\,\,t\ge 0 \,\, \forall \, x \in \mathbb{R} .
  \end{cases} $$
Taking $ x=y-c_{1}\,t$, yields
 $$\begin{cases}
   \phi_{11}(y) \ge \phi_{21}(y+(c_{2}-c_{1})t)\quad \forall t \ge 0, \forall y \in \mathbb{R} \\
   \phi_{12}(y) \ge \phi_{22}(y+(c_{2}-c_{1})t)\quad \forall t \ge 0, \forall y \in \mathbb{R}.
  \end{cases} $$
Using that $ c_{1} < c_{2} $ and passing to the limit $ t \to +\infty $, we get 
 $$\begin{cases}
   \phi_{11}(y) \ge \phi_{21}(+\infty)=1 \,\,\forall y \in \mathbb{R} \\
   \phi_{12}(y) \ge \phi_{22}(+\infty)=1 \,\,\forall y \in \mathbb{R}.
  \end{cases} $$
But $  \phi_{11}(-\infty)=0 $ and $  \phi_{12}(-\infty)=0 $, hence a contradiction. Therefore  $ c_{1} \ge c_{2} $.
Similarly, we prove that  $  c_{2} \ge c_{1} $. Thus $ c_{1}=c_{2} $.

\end{proof}

%%%%%%%%%%%%%%%%%%%%%%%%%%%%%%%%
%%%%%%%%%%%%%%%%%%%%%%%%%%%%%%%%
\section{Uniqueness of the profile }\label{sec:5}
%%%%%%%%%%%%%%%%%%%%%%%%%%%%%%%%
This section is devoted to the proof of the uniqueness of the profiles (under assumption $ (D\pm) $) using tow different types of strong maximum principle.
%%%%%%%%%%%%%%%%%%%%%%%%%%%%%%%%
\subsection{Different types of strong maximum principle}
\begin{lem}[Half Strong Maximum Principle]\label{lemma 9}
Let  $F:[0,1]^{N+1} \to \mathbb{R} $ satisfying assumption  $(A)$ and let  $(\phi_{11}, \phi_{12})$ and $(\phi_{21}, \phi_{22})$  be 
respectively a viscosity sub and super-solution of (\ref{eq 5}), with $\phi_{11},\phi_{12},\phi_{21},\phi_{22}:\R\to [0,1]$. We assume that 
$$ \begin{cases}
\phi_{21} \ge \phi_{11} \quad \mbox{on} \quad \mathbb{R}\\
\phi_{22} \ge \phi_{12}\quad \mbox{on} \quad\mathbb{R}\\
\phi_{21}(0)=\phi_{11}(0) \\
\phi_{22}(0)=\phi_{12}(0). 
 \end{cases}$$
If $ c >0 $ (resp. $ c < 0 $), then 
$$\begin{cases}
   \phi_{11}= \phi_{21}\quad \mbox{ for\,\,all} \quad x\,\le\,0\,\,(resp.\,x\,\ge\,0)\\
   \phi_{12}= \phi_{22}\quad \mbox{ for\,\,all} \quad x\,\le\,0\,\,(resp.\,x\,\ge\,0).\\
  \end{cases}$$
\end{lem}
\begin{proof}
We do the proof in the case where $c>0$. By contradiction, assume that there exists $x_0<0$ such that 
$$\phi_{21}(x_0)>\phi_{11}(x_0)\quad {\rm or}\quad \phi_{22}(x_0)>\phi_{12}(x_0).$$
Let $ w_{1}(x)=\phi_{21}(x)-\phi_{11}(x) $ and $ w_{2}(x)=\phi_{22}(x)-\phi_{12}(x) $. 
A simple computation gives that
\begin{equation}\label{eq 62}
\left\lbrace
\begin{array}{lll}
c\,w_{1}'(x) &\ge& \alpha_{0}\left(w_{2}(x)-w_{1}(x)\right)\\
c\,w_{2}'(x) &\ge& 2(F((\phi_{21}(x+r_{i}))_{i=0,...,N})-F((\phi_{11}(x+r_{i}))_{i=0,...,N})+ \alpha_{0} \left(w_{1}(x)-w_{2}(x)\right).
 \end{array}
\right.
\end{equation}
Using that $ F $ is non-decreasing w.r.t. $ X_{i} $ for all $ i \neq 0 $, we get
$$ \left\{\begin{array}{lll}
c\,w_{1}'(x)& \ge& \alpha_{0}\left(w_{2}(x)-w_{1}(x)\right),\\
c\,w_{2}'(x) &\ge& 2\left(F(\phi_{11}(x)+w_{1}(x), (\phi_{11}(x+r_{i}))_{i=1,...,N})-F(\phi_{11}(x),(\phi_{1}(x+r_{i}))_{i=1,...,N}\right))\\
&&+ \alpha_{0} \left(w_{1}(x)-w_{2}(x)\right)
\end{array}\right.$$
Let $ w(x)=w_{1}(x)+w_{2}(x) $. Then
$$ c\,w'(x) \ge 2\left(F(\phi_{11}(x)+w_{1}(x), (\phi_{11}(x+r_{i}))_{i=1,...,N})-F(\phi_{11}(x),(\phi_{11}(x+r_{i}))_{i=1,...,N}\right)) $$
Since $ F $ is globally Lipschitz continuous (we denote by $L$ its Lipschitz constant), we have 
\begin{equation}\label{eq 64}
w'(x) \ge -\dfrac{2\,L}{c}\,\,w_{1}(x)  \ge -\dfrac{2\,L}{c}\,\,w(x).
\end{equation}\medskip

We note that  $ y(x):= w(x_{0})\,\,\exp{\left(-\dfrac{2\,L\,(x-x_{0})}{c}\right)}$ satisfied (\ref{eq 64}) for all $ x_{0} \in \mathbb{R} $. Using the comparison principle,
we deduce that 
\begin{equation}\label{eq:1000}
 w(x) \ge   w(x_{0})\,\,\exp{\left(-\dfrac{2\,L\,(x-x_{0})}{c}\right)}\,\,\mbox{for\,\,all}\,\,x \ge x_{0}. 
 \end{equation}
Since $ w_{1}(x_{0}) > 0 $  or $ w_{2}(x_{0}) > 0$, we have
$$ w(x_{0}) > 0  . $$ 
This implies that
$$ w(x) > 0  \,\,\mbox{for\,\,all}\,\,x\,\ge\,x_{0}. $$  
In particular, for $x=0$, we get
$$ w_{1}(0) > 0\quad \mbox{or}\quad  w_{2}(0) > 0,
$$
i.e.
$$ 
  \phi_{21}(0) > \phi_{11}(0) \quad \mbox{or}\quad
  \phi_{22}(0) > \phi_{12}(0), 
$$
which is a contradiction.

\end{proof}
%----------------------------------------------------------
We now use Lemma \ref{lemma 9} in order to get a Strong Maximum Principle under assumption $ (D\pm)$ $ii)$.
\begin{lem}[Strong Maximum Principle under  $ (D\pm)$ $ii)$]\label{lemma 10}
Let $F:[0,1]^{N+1}  \to \mathbb{R}$ satisfying $(A)$. Let $ (\phi_{11}, \phi_{12}) $ and $ (\phi_{21}, \phi_{22}) $, with $\phi_{11},\phi_{12},\phi_{21},\phi_{22}:\R\to [0,1]$,  be respectively a viscosity sub
and super-solution of (\ref{eq 5}) such that 
$$ \begin{cases} 
\phi_{21} \ge \phi_{11} \quad \mbox{on} \quad \mathbb{R}\\
\phi_{22} \ge \phi_{12} \quad \mbox{on} \quad \mathbb{R}\\
 \phi_{21}(0)=\phi_{11}(0) \\
 \phi_{22}(0)=\phi_{12}(0). 
\end{cases} $$
a) If $F$ is increasing w.r.t. $ X_{i_{0}} $  for a certain $ i_{0} \neq 0 $ then 
$$\begin{cases}
\phi_{21}(k\,r_{i_{0}})=\phi_{11}(k\,r_{i_{0}}) \quad \mbox{for\,\,all} \quad k \in \mathbb{N} \\
\phi_{22}(k\,r_{i_{0}})=\phi_{12}(k\,r_{i_{0}}) \quad \mbox{for\,\,all} \quad k \in \mathbb{N}.
\end{cases}
$$
b) If we suppose moreover that $ F $ satisfies $ (D+)\,ii)$ if $ c > 0 $ or $ (D-)\,ii)$ if $ c < 0 $ then
$$\begin{cases}
\phi_{21}(x)=\phi_{11}(x) \quad \mbox{for\,\,all} \quad x\in \mathbb{R} \\
\phi_{22}(x)=\phi_{12}(x) \quad \mbox{for\,\,all} \quad x \in \mathbb{R}. 
\end{cases}
$$
\end{lem}
%------------------------------------------
\begin{proof}
a) We assume for simplicity of notation that $ i_{0}=1 $. As in the proof of Lemma \ref{lemma 9}, we define $ w_{1}(x)=\phi_{21}(x)-\phi_{11}(x) $ and
 $ w_{2}(x)=\phi_{22}(x)-\phi_{12}(x) $ which satisfy
\begin{equation}\label{eq 65}
\left\lbrace
\begin{array}{lcl}
c\,w_{1}'(x) \ge \alpha_{0}(w_{2}(x)-w_{1}(x))\\
c\,w_{2}'(x) \ge 2(F((\phi_{21}(x+r_{i}))_{i=0,...,N})-F((\phi_{11}(x+r_{i}))_{i=0,...,N})) + \alpha_{0}(w_{1}(x)-w_{2}(x)).
 \end{array}
\right.
\end{equation}
Using that $w_{1}(0)=0, w_{2}(0)=0$ and $ w_{1}, w_{2} \ge 0$ on $ \mathbb{R}$ (hence $0$ is a point of minimum of $w_1$ and $w_2$), we deduce that 
$$\begin{cases}
 0 \ge \alpha_{0}(w_{2}(0)-w_{1}(0))\\
0 \ge 2(F((\phi_{21}( r_{i}))_{i=0,...,N})-F((\phi_{11}(r_{i}))_{i=0,...,N})) + \alpha_{0}(w_{1}(0)-w_{2}(0)).
\end{cases}$$
Thus 
$$ 0 \ge 2\left(F((\phi_{21}( r_{i}))_{i=0,...,N})-F((\phi_{11}(r_{i}))_{i=0,...,N})\right). $$
Using the fact that $ \phi_{21}(0)= \phi_{11}(0) $ and that $ F $  is monotone w.r.t. $ X_{i} $ for all $ i \neq 0 $, we get 
$$ F((\phi_{21}(r_{i}))_{i=0,...,N})=F((\phi_{11}(r_{i}))_{i=0,...,N}).$$
Since $ F $ is increasing w.r.t. $ X_{1} $, we deduce that 
$$ \phi_{21}(r_{1})=\phi_{11}(r_{1}), $$
i.e. $w_1(r_1)=0$. Hence $r_1$ is a point of minimum of $w_1$. The first equation of \eqref{eq 65} then implies
$$0\ge w_2(r_1)-w_1(r_1)=w_2(r_1).$$
Since $w_2\ge0$, we deduce that $w_2(r_1)=0$, i.e. 
$$\phi_{22}(r_{1})=\phi_{12}(r_{1}).$$
Repeating the above argument replacing $0$ by $r_{1}$, we get that
$$ \phi_{21}(k\,r_{1})=\phi_{11}(k\,r_{1}) \quad {\rm and}\quad \phi_{22}(k\,r_{1})=\phi_{12}(k\,r_{1}) \quad \mbox{for\,\,all} \quad k\,\in\,\mathbb{N}. $$
\medskip

\noindent b) We assume that $ c > 0 $ and that $ F $ satisfies  $ (D+) $ $ ii) $ (the other case where $ c < 0 $ being similar). 
By contradiction, we suppose that there exists $ x \in \mathbb{R} $, such that 
$$
\phi_{21}(x) > \phi_{11}(x) \quad{\rm or}\quad
\phi_{22}(x) >\phi_{12}(x) .$$ 
Let $ k \in \mathbb{N} $ big enough such that  $ k\,r_{i_{+}} > x  $. Using Lemma \ref{lemma 9}, and the fact 
that 
$$\begin{cases}
 \phi_{11}(k\,r_{i_{+}})=\phi_{21}(k\,r_{i_{+}}) \\
 \phi_{12}(k\,r_{i_{+}})=\phi_{22}(k\,r_{i_{+}}). \\
 \end{cases}$$
We get that 
$$ \begin{cases}
\phi_{21}(x)=\phi_{11}(x) \\
\phi_{22}(x)=\phi_{12}(x).
 \end{cases}$$
which is a contradiction.

\end{proof}
%-------------------------------------------------------

 \begin{lem}[Comparison Principle under $ (D \pm) $ $i)$]\label{Lem 12}
We assume that $ c > 0 $ (resp. $ c< 0 $) and let $ F $satisfying $ (A) $ and  $ (D+) $ $ i) $ (resp. $ (D-) $ $ i) $).
Let $ (\phi_{11}, \phi_{12}) $ and $ (\phi_{21}, \phi_{22}) $ be two solutions of (\ref{eq 5}), with $\phi_{11},\phi_{12},\phi_{21},\phi_{22}:\R\to [0,1]$. We assume that $\phi_{11}$, $\phi_{21}\in C^2$ and $\phi_{12}$, $\phi_{22}\in C^1$ and that
$$ \phi_{11}(0)=\phi_{21}(0)\quad {\rm and}\quad \phi_{12}(0)=\phi_{22}(0).$$
Suppose moreover that
$$\begin{cases}
   \phi_{21}(x) \ge \phi_{11}(x) \quad \mbox{on} \quad [-r^{*},0] \quad (\mbox{resp.\,\,on}\,\,[0, r^{*}] )\\
   \phi_{22}(x) \ge \phi_{12}(x) \quad \mbox{on} \quad [-r^{*},0] \quad (\mbox{resp.\,\,on}\,\,[0, r^{*}] )
  \end{cases}$$
then  
$$\begin{cases}
   \phi_{21}(x) \ge \phi_{11}(x) \quad \mbox{ for\,all} \quad x\,\ge\,-r^{*} \,\,(\mbox{resp.}\,\,x\,\le\,r^{*} ) \\
    \phi_{22}(x) \ge \phi_{12}(x) \quad \mbox{for \,all} \quad x\,\ge\,-r^{*} \,\,(\mbox{resp.}\,\,x\,\le\,r^{*}).\\
  \end{cases}$$
\end{lem}
%--------------------------------------
\begin{proof}
We assume that $ c >0 $ (the case $c < 0$ being similar). 
We define the functions $ w_{1}(x)=\phi_{11}(x)-\phi_{21}(x) $ and  $ w_{2}(x)=\phi_{12}(x)-\phi_{22}(x) $
which satisfy
\begin{equation}\label{eq:30}\begin{cases}
   c\,w_{1}'(x) = \alpha_{0}(w_{2}(x)-w_{1}(x))\\
   c\,w_{2}'(x) = 2(F((\phi_{11}(x+r_{i}))_{i=0,...,N})-F((\phi_{21}(x+r_{i}))_{i=0,...,N})) + \alpha_{0}(w_{1}(x)-w_{2}(x)).
  \end{cases} 
  \end{equation}
  By the first equation, we then deduce that 
  $$w_2'= w_1'+\frac c {\a_0}w_1''.$$ 
  The second equation then implies that
  $$\frac {c^2} {\a_0}w_1''+2c w_1'= 2(F((\phi_{11}(x+r_{i}))_{i=0,...,N})-F((\phi_{21}(x+r_{i}))_{i=0,...,N})).$$
Since  $ \phi_{11} \le \phi_{21} $  on $ [-r^{*},0] $  and $ r_{i} \le 0 $ for all $ i \neq 0 $,
 then for all $ x \in [0, \underset{i \neq 0}{\min}(-r_{i})]$, we have $ \phi_{11}(x+r_{i}) \le \phi_{21}(x+r_{i})$ for $ i \neq 0$. This implies that
\begin{align*}
\frac {c^2} {\a_0}w_1''+2c w_1'\le& 2(F(\phi_{21}(x)+w_1(x),(\phi_{21}(x+r_{i}))_{i=1,...,N})-F(\phi_{21}(x),(\phi_{21}(x+r_{i}))_{i=1,...,N}))\\
\le& {2\,L}\,\,|w_{1}(x)| 
\end{align*}
where $L$ is the Lipschitz constant of $F$.
Moreover, $w_1(0)=0$, $w_1'(0)=\frac {\alpha_0}c (w_2(0)-w_1(0))=0$ and $ y=0 $ is a solution of $\frac {c^2} {\a_0}y''+2c y'={2L} y $, then using the comparison principle, we deduce that
$$ w_1 \le 0 \quad \mbox{for\,\,all} \quad x \in [0, \min_{i \neq 0}(-r_{i})]$$
i.e.
$$ \phi_{11} \le \phi_{21} \quad \mbox{for\,\,all} \quad x \in [0, \min_{i \neq 0}(-r_{i})].$$
Using the second equation of \eqref{eq:30} and the fact that  $ \phi_{11}(x+r_{i}) \le \phi_{21}(x+r_{i})$ for $ i \neq 0$ for all $ x \in [0, \underset{i \neq 0}{\min}(-r_{i})]$, we deduce that
\begin{align*}
  c\,w_{2}'(x) \le&  2(F(\phi_{21}(x)+w_1(x),(\phi_{21}(x+r_{i}))_{i=1,...,N})-F(\phi_{21}(x),(\phi_{21}(x+r_{i}))_{i=1,...,N}))\\
  & + \alpha_{0}(w_{1}(x)-w_{2}(x))\\
  =&G(w_1(x))-G(0)-\a_0 w_2(x)
  \end{align*}
  where $G(t)=2F(\phi_{21}(x)+t,(\phi_{21}(x+r_{i}))_{i=1,...,N}) +\alpha_0 t$. Using that $G$ is non-decreasing (see \eqref{eq:monV0}) and $w_1(x)\le0$, we deduce that
  $$ c\,w_{2}'(x) \le-\a_0 w_2(x).$$
Using again that $ w_2(0)=0 $ and $ y=0 $ is a solution of $ w'(x)= \frac {-\a_0}c w_2(x) $, we deduce by the comparison principle that
$$ w_2 \le 0 \quad \mbox{for\,\,all} \quad x \in [0, \min_{i \neq 0}(-r_{i})].$$
We repeat the above argument several times, each on the new extended interval.
We deduce that 
$$ \begin{cases}
  \phi_{11} \le \phi_{21}\quad \mbox{for\,\,all}\quad x \ge -\,r^{*} \\
   \phi_{12} \le \phi_{22}\quad \mbox{for\,\,all}\quad x \ge -\,r^{*}.
   \end{cases}$$
\end{proof}
%----------------------------------------------
We use Lemma \ref{lemma 9} and Lemma  \ref{Lem 12} in order to prove  the Strong Maximum principle under $(D\pm)$ $ i)$.
\begin{lem}[Strong Maximum Principle under $(D\pm)$ $ i)$]\label{lem 14}
We assume that $ c >0 $ (resp. $ c<0 $) and $ F $ satisfies  $(A)$ and $(D+)\,i) $ (resp. $ (D-)\,i)) $. Let $(\phi_{11}, \phi_{12})$ and $(\phi_{21}, \phi_{22} ) $ 
be respectively a viscosity sub and a super-solution of (\ref{eq 5}), with $\phi_{11},\phi_{12},\phi_{21},\phi_{22}:\R\to [0,1]$. We assume that $\phi_{11}$, $\phi_{21}\in C^2$ and $\phi_{12}$, $\phi_{22}\in C^1$ and that
$$ \begin{cases}
\phi_{21} \ge \phi_{11} \quad \mbox{on} \quad \mathbb{R}\\
\phi_{22} \ge \phi_{12}\quad \mbox{on} \quad\mathbb{R}\\
\phi_{21}(0)=\phi_{11}(0) \\
\phi_{22}(0)=\phi_{12}(0). 
 \end{cases}$$
Then 
$$\begin{cases} 
 \phi_{11}(x)=\phi_{21}(x) \quad \mbox{for\,\,all} \quad x \quad \mbox{in} \quad \mathbb{R}\\
 \phi_{12}(x)=\phi_{22}(x) \quad \mbox{for\,\,all} \quad x \quad \mbox{in} \quad \mathbb{R}.\\
\end{cases}$$
\end{lem}
%-------------------------------------------
\begin{proof}Let $ c >0 $ . Using Lemma \ref{lemma 9}, we deduce that 
$$ \begin{cases}
 \phi_{11}=\phi_{21} \quad \mbox{for\,\,all}\quad x \le 0 \\
 \phi_{21}=\phi_{22} \quad \mbox{for\,\,all}\quad x \le 0. 
\end{cases}$$
By Lemma \ref{Lem 12}, we then deduce that 
$$ \begin{cases}
 \phi_{11} \ge \phi_{21}\quad \mbox{for\,\,all} \quad x \ge -r^{*} \\
 \phi_{12} \ge \phi_{22}\quad \mbox{for\,\,all} \quad x \ge -r^{*}.   
\end{cases}$$
which gives the result.

\end{proof}
%---------------------------------
\begin{lem}[Ordering two solutions of (\ref{eq 8}) up to translation]\label{Lem 15} 
We assume that $ c \neq 0 $ and let $ F: [0,1]^{N+1} \to \mathbb{R} $ satisfiyng $ (A)$ and $ (C) $. Let $(\phi_{11}, \phi_{12})$ and 
 $(\phi_{21}, \phi_{22})$ be respectively a viscosity sub
and super-solution of (\ref{eq 5}), with $\phi_{11},\phi_{12},\phi_{21},\phi_{22}:\R\to [0,1]$. There exists a shift $ a^{*} \in \mathbb{R}$ and some $ x_{0}
 \in [-r^{*}, r^{*}]$ such that  
 $$\begin{cases}
   \phi_{21}^{a^{*}} \ge \phi_{11}\quad \mbox{on} \quad \mathbb{R} \\
     \phi_{22}^{a^{*}} \ge \phi_{21} \quad \mbox{on} \quad \mathbb{R} \\
 \phi_{21}^{a^{*}}(x_{0})= \phi_{11}(x_{0}) \\
   \phi_{22}^{a^{*}}(x_{0})= \phi_{12}(x_{0}),
  \end{cases}$$
 where 
$$\begin{cases}
   \phi_{21}^{a^{*}}(x)=\phi_{21}(x+a^{*}) \\
   \phi_{22}^{a^{*}}(x)=\phi_{22}(x+a^{*}).
  \end{cases}$$
\end{lem}
%------------------------------
\begin{proof}
The idea of the proof is to translate $ (\phi_{21}, \phi_{22}) $ and then to compare it with $ (\phi_{11}, \phi_{12}) $.

\paragraph{Step 1: Family of solutions above $(\phi_{11}, \phi_{12}) $.} 
For  $ a \in \mathbb{R} $, we define
$$ \begin{cases}
   \phi_{21}^{a}(x)=\phi_{21}(x+a) \\
   \phi_{22}^{a}(x)=\phi_{22}(x+a). 
   \end{cases}$$
For some $ a > 0 $ large enough, we have 
$$ \begin{cases}
   \phi_{21}^{\bar{a}} \ge \phi_{11} \quad \mbox{on} \quad [-2r^{*}, 2r^{*}] \quad \mbox{for\,\,all} \quad \bar{a} > a \\
   \phi_{22}^{\bar{a}} \ge \phi_{12} \quad \mbox{on} \quad [-2r^{*}, 2r^{*}] \quad \mbox{for\,\,all} \quad \bar{a} > a. \\
   \end{cases}$$
 Using the comparison principle (Theorem \ref{th 2} and Corollary \ref{cor 2} ),  we deduce that for all $ \bar{a} \ge a $, we have 
$$\begin{cases}
      \phi_{21}^{\bar{a}} \ge \phi_{11} \quad \mbox{on} \quad \mathbb{R} \\
      \phi_{22}^{\bar{a}} \ge \phi_{12} \quad \mbox{on} \quad \mathbb{R}.
  \end{cases}$$
  
\paragraph{Step 2: There exists $ a^{*} $ such that $ \phi_{21}^{a^{*}} $ and $\phi_{11}$ touch at 
 $ x_{0} \in [-r^{*}, r^{*}] $ , $ \phi_{22}^{a^{*}} $ and $\phi_{12}$ touch at $ x_{0} \in [-r^{*}, r^{*}] $.} 
Let 
$$ \begin{cases}
 a^{*}_{1}=\inf \{a\,\,\in\,\,\mathbb{R}, \,\phi_{21}^{ \bar{a} }\,\ge\,\phi_{11}\quad\mbox{on}\quad \mathbb{R} \quad\mbox{for\,\,all}\quad \bar{a} \,\ge\,a \}  \\ 
 a^{*}_{2}=\inf \{a\,\,\in\,\,\mathbb{R}, \,\phi_{22}^{ \bar{a} }\,\ge\,\phi_{12}\quad\mbox{on}\quad \mathbb{R} \quad\mbox{for\,\,all}\quad \bar{a} \,\ge\, a  \} . 
   \end{cases}$$
We set $ a^{*}=max(a^{*}_{1}, a^{*}_{2}) $. 
We define $ k_{1}(x)=\phi_{21}^{a^{*}}(x)-\phi_{11}(x) $ and
 $ k_{2}(x)=\phi_{22}^{a^{*}}(x)-\phi_{12}(x) $ which satisfy
 \begin{equation}\label{eq 65bis}
\left\lbrace
\begin{array}{lcl}
c\,k_{1}'(x) \ge \alpha_{0}(k_{2}(x)-k_{1}(x))\\
c\,k_{2}'(x) \ge 2(F((\phi_{21}^{a^{*}}(x+r_{i}))_{i=0,...,N})-F((\phi_{11}(x+r_{i}))_{i=0,...,N})) + \alpha_{0}(k_{1}(x)-k_{2}(x)).
 \end{array}
\right.
\end{equation}
We now prove that $a_2^*=a_1^*$. By contradiction, assume that $a_2^*\ne a_1^*$
If $ a_{2}^{*} > a_{1}^{*} $, we have
$$\begin{cases}
 \phi_{21}^{a^{*}}(x) > \phi_{11}(x) \\
 \phi_{22}^{a^{*}}(x) \ge\phi_{12}(x) \\
 \phi_{22}^{a^{*}}(x_0) =\phi_{12}(x_0). 
  \end{cases} $$
Since F is non-decreasing in $ X_{i} $ for $ i \neq 0$, we have
\begin{align*}
F((\phi_{21}^{a^{*}}(x+r_{i}))_{i=0,...,N}) \ge& F(\phi_{21}^{a^{*}}(x),(\phi_{11}(x+r_{i}))_{i=1,...,N})\\
 \ge& F(\phi_{11}(x)+k_{1}(x),(\phi_{11}(x+r_{i}))_{i=1,...,N})
 \end{align*}
Thus (since $x_0$ is a point of minimum of $k_2$ and $k_2(x_0)=0$)
\begin{equation}\label{eq 66}
 0 \ge 2(F(\phi_{11}(x_0)+k_{1}(x_0),(\phi_{11}(x_0+r_{i}))_{i=1,...,N})-F((\phi_{11}(x_0+r_{i}))_{i=0,...,N})) + \alpha_{0}k_{1}(x_0)
\end{equation}
We define $ G:\mathbb{R} \to \mathbb{R}$ by
 \begin{equation}\label{eq 67}
  G(y)=2\,F(\phi_{11}(x_{0})+y,(\phi_{11}(x_{0}+r_{i}))_{i=1...N})+\alpha_{0}\,y.
 \end{equation}
 Then
 $$0  \ge (G(k_{1}(x_{0})-G(0)).$$
 But $ k_{1}(x_{0}) > 0 $ and, by assumption \eqref{eq:monV0},  $ G'(y)>0 $, so we get a contradiction.\medskip
 
If $ a_{1}^{*} > a_{2}^{*} $, then 
$$\begin{cases}
 \phi_{22}^{a^{*}}(x) >\phi_{12}(x) \\
 \phi_{21}^{a^{*}}(x) \ge \phi_{11}(x)  \\
 \phi_{21}^{a^{*}}(x_0) = \phi_{11}(x_0)  .
  \end{cases} $$
Since $ k_{1}(x_{0})=0$ and $ k_{1}(x)\ge0$, we deduce that $x_0$ is a point of minimum of $ k_{1}$ and so by the first equation of (\ref{eq 65bis}), we get
$$
0 \ge \alpha_{0}\,k_{2}(x_{0}) >0
$$
wich is a contradiction.\medskip

Thus $ a_{1}^{*}=a_{2}^{*} =a^*$  and so
$$\begin{cases}
 \phi_{21}^{a^{*}}(x) \ge \phi_{11}(x)  \\
 \phi_{22}^{a^{*}}(x) \ge\phi_{12}(x) \\
   \phi_{21}^{a^{*}}(x_{0})= \phi_{11}(x_{0}) \\
   \phi_{22}^{a^{*}}(x_{0})= \phi_{12}(x_{0}).
  \end{cases}$$

\end{proof}
%----------------------------
\begin{lem}[Monotonicity of the profiles]\label{le20}
Assume that $c>0$ (resp. $c<0$) and let $F:[0,1]^{N+1}\to\R$ satisfying $(A)$,  $(C)$ and $(D+)$ $i)$ or $ii)$.
Let $\phi_1,\;\phi_2:\R\to[0,1]$ be a solution of (\ref{eq 5}). Then $\phi_1$ and $\phi_2$ are increasing on $\R.$

\end{lem}
%---------------------------
\begin{proof}
Assume that $c>0$ (the proof when $c<0$ being similar) and let $(\phi_1,\phi_2)$ be a solution of (\ref{eq 5}).

\paragraph{Step $1$: $(\phi_1,\phi_2)$ are non-decreasing.}
The goal is to show that $\phi_1(x+a)\geq\phi_1(x)$  and $\phi_2(x+a)\geq\phi_2(x)$ for all $a\geq0.$ As in the proof of Lemma \ref{Lem 15}, we
deduce that for $a\geq0$ large enough and for all $\overline{a}\geq a,$ we have
$$\phi_1^{\overline{a}}(x):=\phi_1(x+\overline{a})\geq\phi_1(x)
\quad{\rm and}\quad\phi_2^{\overline{a}}(x):=\phi_2(x+\overline{a})\geq\phi_2(x)
\quad\mbox{on}\quad[-2r^{*},2r^{*}].$$
Thus using the comparison principle (Theorem \ref{th 2} and Corollary \ref{cor 2}),
we deduce that for all $\overline{a}\geq a,$ we have 
$$\phi_1^{\overline{a}}(x)\geq\phi_1(x)
\quad {\rm and}\quad \phi_1^{\overline{a}}(x)\geq\phi_1(x)
\quad\mbox{on}\quad\R.$$
Let
$$ \begin{cases}
 a^{*}_{1}=\inf \{a\,\,\in\,\,\mathbb{R}, \,\phi_{21}^{ \bar{a} }\,\ge\,\phi_{11}\quad\mbox{on}\quad \mathbb{R} \quad\mbox{for\,\,all}\quad \bar{a} \,\ge\,a \}  \\ 
 a^{*}_{2}=\inf \{a\,\,\in\,\,\mathbb{R}, \,\phi_{22}^{ \bar{a} }\,\ge\,\phi_{12}\quad\mbox{on}\quad \mathbb{R} \quad\mbox{for\,\,all}\quad \bar{a} \,\ge\, a  \} . 
   \end{cases}$$
As in the proof of Lemma \ref{Lem 15}, we can prove that $a_1^{*}=a_2^*=a^*.$
We want to prove that $a^{*}=0.$ By definition of $a^*$,  there exists some $x_{0}$ such that 
\begin{equation}\label{eq:004}
\left\{
\begin{aligned}
&\phi_1^{a^{*}}\geq\phi_1\quad\mbox{on}\quad\R\\
&\phi_2^{a^{*}}\geq\phi_2\quad\mbox{on}\quad\R\\
&\phi_1^{a^{*}}(x_{0})=\phi_1(x_{0})\\
&\phi_2^{a^{*}}(x_{0})=\phi_2(x_{0}).
\end{aligned}
\right.
\end{equation}
Then, using the Strong Maximum Principle Lemma \ref{lemma 10} or Lemma \ref{lem 14} (note that, since $c\ne 0$, $\phi_1,\; \phi_2\in C^1$ and so the first equation of \eqref{eq 5} gives that $\phi_1\in C^2$, then we can apply Lemma \ref{lem 14}), we get that $\phi_1^{a^*}=\phi_1$, i.e., $\phi_1$
is periodic of period $a^*$. But $\phi_1(-\infty)=0$ and $\phi_1(+\infty)=1$, thus $a^*=0$.

\paragraph{Step $2$: $(\phi_1,\phi_2)$ are increasing.}
Let $a>0,$ we want to show that $\phi_1(x+a)>\phi_1(x)$ and $\phi_2(x+a)>\phi_2(x)$. From Step $1,$ we have $\phi_1(x+a)\geq\phi_1(x)$ and $\phi_2(x+a)\geq\phi_2(x)$.
Assume that there exists $x_{0}$ such that 
$$\phi_1(x_{0}+a)=\phi_1(x_{0})\quad{\rm or}\quad \phi_2(x_{0}+a)=\phi_2(x_{0}).$$
As in the proof of Lemma \ref{Lem 15}, we can prove that
$$\phi_1(x_{0}+a)=\phi_1(x_{0})\quad{\rm and}\quad \phi_2(x_{0}+a)=\phi_2(x_{0}).$$
Using the Strong Maximum Principle (Lemma \ref{lemma 10} or Lemma \ref{lem 14}), we get that $a=0,$ which is a contradiction. Thus
$$\phi_1(x+a)>\phi_1(x)
\quad{\rm and}\quad \phi_2(x+a)>\phi_2(x)
\quad\mbox{on}\quad\R\quad\mbox{for any}\quad a>0.$$
\end{proof}
%---------------------------------------------
\begin{proof}[Proof of Theorem \ref{th 2bis}]
The proof of the uniqueness of the velocity is a direct consequence of Proposition \ref{pro 4}. For the uniqueness of the profiles, it suffices to use Lemma \ref{Lem 15} and the Strong Maximum Principle (Lemma \ref{lemma 10} or Lemma \ref{lem 14}). Note that, since $c\ne 0$, $\phi_1,\; \phi_2\in C^1$ and so the first equation of \eqref{eq 5} gives that $\phi_1\in C^2$, then we can apply Lemma \ref{lem 14}. Finally the strict monotonicity of $\phi_1$ and $\phi_2$ follows from Lemma \ref{le20}.
\end{proof}

\vspace{10mm}

\noindent{\bf ACKNOWLEDGMENTS}
%%%%%%%%%%%%%%%%%%%%%%%%%%

 The first author was partially supported by ANR AMAM (ANR 10-JCJC 0106), ANR IDEE (ANR-2010-0112-01) and  ANR HJNet (ANR-12-BS01-0008-01).

\end{document}